\documentclass[11pt]{amsart}
\usepackage{amsmath}
\usepackage{setspace}
\usepackage{graphicx}
\usepackage{amssymb,amsthm}
\usepackage{hyperref}
\usepackage{esint}
\usepackage{epstopdf}
\usepackage{color}
\usepackage{url}
\usepackage[top=1in, left=1in, right=1in, bottom=1in, marginpar=2cm]{geometry}

\newtheorem{theorem}{Theorem}[section]
\newtheorem{lemma}[theorem]{Lemma}

\newtheorem{remark}[theorem]{Remark}

\def\R{{\mathbb R}}

\def\h1{\> d{\mathcal H}^1}

\def\a{\alpha}

\def\e{\varepsilon}
\def\d{\delta}

\def\t{\tau}

\def\w{\omega}

\def\1{\left(}
\def\2{\right)}
\def\3{\left\{}
\def\4{\right\}}
\def\8{\infty}
\def\sm{\setminus}
\def\ss{\subseteq}

\usepackage{enumerate}
\usepackage{marginnote}

\newcommand{\ufree}{u_{\text{free}}}

\begin{document}

\title{A complete classification of singularities for radial solutions of the Stefan problem}

\author{Dennis Kriventsov}
\address[Dennis Kriventsov]{Rutgers University, Piscataway, NJ}
\email{dnk34@math.rutgers.edu}

\author{Georg S. Weiss}
\address[Georg S. Weiss]{Faculty of Mathematics, University of Duisburg-Essen, Germany}
\email{georg.weiss@uni-due.de}

\begin{abstract}
We investigate singularities of radial solutions of the classical two-phase and one-phase Stefan problem. In the remarkable result \cite{HV}, M.A. Herrero and J.J.L. Velazquez show existence of a solution
such that the free boundary described by $g(t)$ satisfies in dimension $n\geq 3$
$$g(t)=c_n \sqrt{t} |\log t|^{\frac{-1}{n - 2}} (1 + o(1)), t\to 0,$$
and they conjecture (\cite[p. 2]{HV}) that these asymptotics hold in general.
Here we prove the conjecture for all radial solutions,
which means {\em complete asymptotic rigidity} in that class. 
  
In order to obtain these precise asymptotics for the two-phase Stefan problem, for which frequency formulas and monotonicity formulas are unknown,
we apply the {\em Potential Reduction to ODE method} introduced in \cite{li_lin_weiss_xie} for the Neumann Bernoulli problem.
The philosophy of that method is that the asymptotic shape of the free surface is governed by a precise ordinary differential
equation. The fact that we can apply it in this paper to an equation of different type with quasilinear principal part,
which is in our regime not the perturbation of a linear operator,
confirms the versatility of the method.
\end{abstract}
\maketitle
\begingroup
\small
\setstretch{0.5}
\tableofcontents
\endgroup
\section{Introduction}

The Stefan problem is the parabolic PDE
\[
	\theta(u)_t - \Delta u = 0
\]
where $\theta$ is given by
\[
	\theta(s) = \begin{cases}
		s + 1, & s > 0 \\
		s, & s \leq 0.
	\end{cases}
\]
Here $u : (T_0, T_1] \times B_R \rightarrow \R$ models the temperature of a fluid undergoing a phase transition, with $\{u \leq 0\}$ representing the solid (``ice'') phase and $\{u > 0\}$ the liquid (``water'') phase. For a \emph{classical solution}, a distributional solution $u$ with a smooth interface and $u$ smooth in each phase, the PDE gives a local condition relating the speed of the interface $\{u = 0\}$ (oriented so that positive speed means melting) and the gradient of $u$ from each side:
\[
	V = \frac{\partial_t u_+}{|\nabla u_+|} = |\nabla u_+| - |\nabla u_-|.
\]
Here the subscripts $u_\pm$ refer to taking derivatives from the $\{u > 0\}$ and $\{u \leq 0\}$ sides, respectively; the derivatives from the two sides will not be the same. The \emph{two-phase} Stefan problem has $u_- < 0$ on the interior of $\{u \leq 0\}$, while the \emph{one-phase} Stefan problem instead has $u = 0$ there and so nonnegative globally, essentially treating a simplified model where the interior temperature of the ice is not tracked.

In this paper we study radial solutions of the two-phase and the one-phase problem, where $u(t, x) = u(t, |x|)$, and in particular their behavior at a \emph{melting singularity}: a time $t$ at which $u(t, 0) = 0$ and $u(t, |x|) > 0$ for $|x| > 0$. If a solution is extended past this time (and excluding boundary effects) it will contain only the water phase and simply solve the heat equation, so a melting singularity corresponds to the extinction of the ice phase at a single point. Radial solutions with a connected ice phase containing $x = 0$, which is the setting we will work in, are classical except at one melting singularity.

For concreteness, let us assume there is a melting singularity at time $t = 0$, and parametrize the interface as $|x| = g(-t)$.

The Stefan model is well-known for not admitting useful self-similar solutions with respect to the parabolic rescaling. Indeed, the nonlinearity is only compatible with $0$-homogeneous self-similar profiles, which leads to only trivial solutions. Therefore the question of the behavior of $g$ as $t\rightarrow 0$ is not a priori clear, and the goal of this paper is to address it.

In \cite{HV}, M.A. Herrero and J.J.L. Velazquez construct an ancient (i.e. backward-entire) radial solution to the Stefan problem with a melting singularity, for which
\[
	g(t) = 
	\begin{cases}
		c_n \sqrt{t} |\log t|^{\frac{-1}{n - 2}} (1 + o(1)), & n \geq 3 \\
		\sqrt{t} e^{- \sqrt{2}/2 \sqrt{|\log t|} + O(1) }, & n = 2.
	\end{cases}
\]
The 2D bound is stated here in a slightly corrected form, see \cite[pg. 3]{HR} for a discussion. Their construction, using the method of matched asymptotic expansions, suggested to them that this was the only possible rate (see the next subsection about the other solutions constructed there), and they state this as an open question on page 2. Our main theorem answers this question:

\begin{theorem}\label{thm:main}
	Let $u : (-T, 0] \times B_R \rightarrow \R$ be a radial solution to the (one or two-phase) Stefan problem with a melting singularity at $(0, 0)$ (more precisely, assume it satisfies (1-4) in Section \ref{sec:prelim}). Then if $n \geq 3$,
	\[
		g(t) =  c_n \sqrt{t}|\log t|^{- \frac{1}{n-2}} (1 + o_t(1)) ,
	\]
	while if $n = 2$,
	\[
		g(t) = \sqrt{t} e^{-\sqrt{-\frac{1}{2}\log t} + o(\sqrt{-\log t})}.
	\]
	The constant $c_n$ is explicit:
	\[
		c_n = \left(\frac{2 (4\pi)^{n/2}}{(n - 2)^2\w_{n-1}}\right)^{\frac{1}{n-2}}.
	\]
\end{theorem}

We conjecture that the method used to prove Theorem \ref{thm:main} can give even finer rates (in the 2D case) as well as an {\em asymptotic expansion for $g$ up to higher order}. We also note that due to the minimal assumptions in the theorem, it can be used
to show {\em existence of solutions} with the above asymptotics: similarly as in \cite{HR} one can show existence of
solutions on bounded
domains which have an extinction point
 (by solving the initial-boundary value problem with appropriate data, using a barrier argument like in Lemma \ref{lem:KP} below to show that it melts in finite time). Then one may apply Theorem 1.1 at the melting point.
These will not be ancient solutions, but we conjecture that further refinements of our method to study behavior in the $t \rightarrow -\infty$ limit can be used to give alternative constructions of ancient solutions as well.

In the next two subsections, we first give a brief overview of the literature on the Stefan problem and the question of melting singularities, and then outline our approach to the proof of Theorem \ref{thm:main}.

\subsection{Existing literature}

The Stefan problem is one of the most-studied moving interface equations, both in connection to the physical model and as an archetype for certain classes of nonlinear parabolic PDE. In describing the literature it is useful to distinguish an additional variant
(not studied in the present paper), 
in which the temperature is allowed to change sign in the water phase and satisfies around these ``phantom interfaces"
the usual heat equation. The authors (see for example \cite{HV}) sometimes describe this 
problem (not satisfying our distributional formulation above)
as ``undercooled,"
not to be confused with the entirely different supercooled/ill-posed Stefan problem.

The existence, uniqueness, and comparison principle for solutions to the Stefan problem have been studied from the point of view of weak or distributional solutions by S. L. Kamenomostskaja in \cite{K}, by A. Friedman in \cite{F}, and by H. Brezis in \cite{B} in a monotone operator formulation. There is also a viscosity formulation of the problem, common in works on regularity theory, for which the global theory and connection to distributional notions is discussed by I. Kim and N. Pozar in \cite{KP11}. A key result for the Stefan problem is the continuity of the temperature $u$: even through a singularity, $u$ remains continuous. In the two-phase case, this was shown by L. A. Caffarelli and L. C. Evans in
\cite{CE} , by E. DiBenedetto in \cite{dB} and by W.P. Ziemer in \cite{Z} in the early 1980s.

In one dimension, the behavior at melting singularities was studied extensively in early references, and by D. G. Aronson and S. Kamin (\cite{AK}) was largely understood: here $g(t) \approx t$, and this can be shown with purely PDE methods. This is particularly clear for the one-phase problem, where, unlike in higher dimensions, the ice disconnects two independently-evolving, smooth, water phases that come into contact transversely. The two-phase behavior is essentially the same despite the small interaction. 

In higher dimensions, as mentioned above, \cite{HV} constructs radial melting solutions with specific profiles for $g$. In fact, the authors construct a countable family of solutions in every dimension with different rates, but to the ``undercooled" variant of the problem,
so strictly speaking only one of their solutions (for each dimension $n$) is a solution of the classical Stefan problem.
The melting rate there corresponds to the number of negative-temperature rings in the water phase, with exactly one of their solutions having no such rings and producing the asymptotics we described above. Several subsequent works, including \cite{HR} by M. Hadzic and P. Raphael and \cite{Z2} by C. Zhang, made refinements and generalizations to the construction, such as checking for stability with respect to radial perturbations.

To the best of our knowledge, there has been no serious study of melting rates for general radial solutions to the two-phase problem at all. While not explicitly stated, it is possible to use the barriers in \cite{KP11} to obtain that $g(t)/\sqrt{t} \rightarrow 0$ (even without the radial assumption). We have also found elementary PDE arguments to get lower bounds which are far from sharp and do not seem to appear in the literature (see Section \ref{sec:pde}).

For the one-phase problem, an entirely different perspective is available via performing the Duvaut transform, i.e. considering the primitive $U(t, x) = \int_{t_0}^t u(s, x) ds$, which solves a parabolic obstacle problem. There has been substantial progress in this direction recently in a series of strong results, using monotonicity formulas and higher-order blow-up arguments in the parabolic scaling; notably A. Figalli, X. Ros-Oton and J. Serra prove in \cite{FRS} a partial regularity theorem that was previously inaccessible. We note that specializing this to the radial case, though, does not actually give much new information about a melting singularity. More relevant to the present work is the important result \cite{FRTL} by G. Fioravanti, X. Ros-Oton, and C. Torres-Latorre, which studies the behavior at melting singularities for general solutions to the one-phase problem via the Duvaut transform and then further PDE arguments. The estimates they obtain are weaker than ours, especially for $n \geq 3$, but apply to nonradial solutions.

\subsection{Ideas from the proof}

Our proof uses the {\em Potential Reduction to ODE method} introduced in \cite{li_lin_weiss_xie} for the Neumann Bernoulli problem ---see also \cite[Section 4]{EFW} on the stationary obstacle problem for a precursor of the method. We work directly with the two-phase Stefan problem, in the physical variables. We begin by rewriting the PDE as $u_t - \Delta u = \mu$, where $\mu$ is a distribution supported on the interface. 
Note that here we divided the quasilinear operator into one which allows for a representation formula
and another, which is {\em not} a perturbation. 
We can then express $u$ as the convolution of $\mu$ with a heat kernel; after a computation this takes the form
\[
	u(-t, x) = - \int_{t}^1 \int_{|y| = g(s)} G(s - t, y - x) g'(s) dS(y) ds + \text{caloric function},
\]
where $G$ is the forward heat kernel, extended by zero for non-positive time. By itself this is not useful, as it only relates the temperature and the interface position in a nonlinear, nonlocal way. We do know, however, that along $\{|x| = g(t)\}$ we have $u(-t, x) = 0$, which ``closes'' the relation. There is still a remaining caloric function, which can be viewed as representing the effects of the initial and boundary data: to make this useful we will subtract two values $u(-t, x) - u(0, 0)$, getting
\[
	\int_{t}^1 \int_{|y| = g(s)} G(s - t, y - x) g'(s) dS(y) ds - \int_{0}^1 \int_{|y| = g(s)} G(s, y) g'(s) dS(y) ds = o_t(1).
\]

We now study these integrals. It will help to assume $g' \geq 0$, which we show is the case for small times in Section \ref{sec:pde}, so the integrands themselves do not change sign. The left integral has a pole at $(s, y) = (t, x)$, and so is much larger there than the other one. In fact we expect that to leading order, the contribution near the pole is the dominant part of the left integral and evaluates to something of order $g^2(t)/t$ after a computation. The right integral, on the other hand, has a pole right at $(0, 0)$, and here the leading-order term is
\[
	A(t) := \int_0^t g^{n-1}(s) g'(s) s^{-n/2} ds,
\]
as the exponential in the heat kernel turns out of order $1$ here. We are left with a perturbed integro-differential equation:
\[
	c(n)\frac{g^2(t)}{t} = (1 + o_t(1)) \> \int_0^t g^{n-1}(s) g'(s) s^{-n/2} ds, \quad n\geq 3,
\]
which can be ``integrated'' to give the rates in Theorem \ref{thm:main}.

Formally, neglecting the asymptotic remainder and differentiating
twice gives the second-order ODE
\[
\begin{aligned}
&\left(2C_n g-t^{1-n/2}g^{n-1}\right)g''
 +\left(2C_n-(n-1)t^{1-n/2}g^{n-2}\right)(g')^2 \\
&\qquad
 +\left(\frac n2\,t^{-n/2}g^{n-1}
        -\frac{4C_n g}{t}\right)g'
 +\frac{2C_n g^2}{t^2}=0,
\end{aligned}
\]
where
\[
    C_n=\frac{(4\pi)^{n/2}}{2(n-2)\w_{n-1}}
       =\frac{n-2}{4}\,c_n^{n-2},
    \qquad n\geq3.
\]

This is, of course, only a simplified sketch of the argument, as there are other contributions to each integral which need to be estimated, the $o_t(1)$ needs to be made more quantitative even in the preliminary differential inequality, the behavior is somewhat different when $n = 2$, and integrating the final relation takes some care. 

However, we believe this is a robust approach to singularity analysis for the two-phase Stefan problem which can extend beyond the radial setting, as well as to other nonlinear PDE. In particular, no use is made of monotonicity formulas or blow-ups.

The organization of the rest of the paper is as follows: in Section \ref{sec:prelim}, we explain our notation and assumptions more carefully. Section \ref{sec:pde} contains a number of PDE estimates on general solutions to gain some rough control over $g$ and $g'$ before entering the main proof. Some of these estimates appear to be new, but follow from fairly standard barrier and Bernstein arguments. In Section \ref{sec:rep}, we justify the potential reformulation above, and then Section \ref{sec:potential} performs the main estimates leading to the perturbed integro-differential equation for $g$. Finally, Section \ref{sec:ode} completes the proof of Theorem \ref{thm:main}. The appendix, Section \ref{sec:appendix}, contains for completeness an argument for the smoothness of the interface away from the melting point, as well as a straightforward PDE estimate for parabolic equations.

\section{Preliminaries} \label{sec:prelim}

We will assume that the space dimension $n \geq 2$. We generally work with functions defined on $Q_R = [-1, 0] \times \bar{B}_R \ss \R \times \R^n$, using the variables $(t, x) \in \R \times \R^n$ for time and space. A \emph{solution} to the Stefan problem on $Q_R$ is a continuous function $u : Q_R \rightarrow \R$ solving
\[
	\theta(u)_t - \Delta u = 0
\]
in the sense of distributions, where $\theta$ is given by
\[
	\theta(s) = \begin{cases}
		s + 1, & s > 0 \\
		s, & s \leq 0.
	\end{cases}
\]
A \emph{classical} solution on a domain is one which has $\partial \{u > 0\}$ a smooth hypersurface separating $\{u > 0\}$ and $\{u \leq 0\}$.

We will deal almost exclusively with \emph{radial} solutions, for which it will be convenient to identify $u(t, x) = u(t, |x|)$, with the latter a function $[-1, 0] \times [0, R] \rightarrow \R$. We define
\[
	g(t) = \inf \{r : u(-t, r) > 0\} \in [0, +\infty].
\]
We will operate under the following assumptions for the main results in this paper. Later, we will discuss verifying these assumptions for general solutions.
\begin{enumerate}
	\item $u \in C(Q_R)$.
	\item $0 < g < R/2$ for $t \in (0, 1]$ and $g : [0, 1] \rightarrow [0, R/2]$ is a continuous curve.
	\item $g(0) = 0$.
	\item $u > 0$ on $\{r > g(-t)\}$, and either $u \equiv 0$ or $u < 0$ on $\{r < g(-t)\}$.
\end{enumerate}
It is well-known that under assumptions (1-4), the solution is classical on $(-1, 0) \times B_R$ (and in particular $g \in C^\infty$ on $(0, 1)$) (see Subsection \ref{sec:appendix} in the Appendix, \cite{ACS} and \cite{koch}).

At times it will be convenient to abuse notation in writing $x = g(t)$ for a distinguished point $x = g(t)e_n$, with $e_n$ a fixed unit vector.

We note here that the property of solving the Stefan problem is invariant under the $0$-ho\-mo\-ge\-neous parabolic rescaling
\[
	u_s(t, x) = u(s^2 t, s x).
\]
This means that if the assumptions are verified in a neighborhood of $(0, 0)$, after a rescaling they will be verified on $Q_R$ for some large $R$. In particular, we are not making any quantitative assumptions on the boundary or initial conditions at this point.

We will use the notation $G(t, x)$ for the value of the forward heat kernel on $(0,\infty)\times \R^n$ with pole at $(0, 0)$:
\[
	G(t, x) = (4\pi)^{-n/2} t^{-n/2} e^{-\frac{|x|^2}{4t}}.
\]

\section{PDE estimates on the temperature and interface}\label{sec:pde}

This section proves several estimates for $g$ using purely PDE-based arguments (i.e. not using the potential representation formula used in Section \ref{sec:rep}).

\subsection{Monotonicity of the interface} 

The goal of this section is to prove that $g' \geq 0$ follows from (1-4). We start by showing that
\[
	\frac{g^2(t)}{t} \rightarrow 0
\]
without making assumptions beyond (1-4). This is a classic result in the one-phase case (see \cite{C78}, or \cite{FRS} for more recent discussion), connected to the fact that blow-ups in the parabolic scaling have time-independent free boundaries. The earliest reference we could find in the two-phase case is \cite[Lemma 4.3]{KP11}. The proof below is essentially taken from there, with minor adaptations to our context.

\begin{lemma}[Adapted from Lemma 4.3 in \cite{KP11} (Kim-Po\v z\'ar)]\label{lem:KP}
	For each $\eta > 0$, there is a $C_\eta$ (depending only on $\eta$ and $n$) such that for $t < s$,
	\[
		g^2(s) \leq (1 + \eta) g^2(t) + C_\eta (s - t) \max_{[-s, 0] \times \overline{B_{2g(s)}}} u.
	\]
	In particular,
	\[
		\lim_{t\searrow 0} \frac{g(t)}{\sqrt{t}} = 0.
	\]
\end{lemma}

\begin{proof}
	Let
	\[
		m(t) = \sup_{[-t, 0] \times B_{2 g(t)}} u > 0;
	\]
	as $u$ is continuous we have that $m(t) \rightarrow 0$ as $t \rightarrow 0$. Set
	\[
		h(r) = \begin{cases}
			0, & r \leq 1, \\
			1, & r \geq 1 + \d, \\
			\frac{r^{2 - n} - 1}{(1 + \d)^{2-n} - 1}, & 1 < r < 1 + \d \text{ and } n \geq 3, \\
			\frac{\log r}{\log (1 + \d)}, & 1 < r < 1 + \d \text{ and } n = 2, \\
		\end{cases}
	\]
	the point being that $h(|x|)$ is harmonic on $B_{1 + \d} \sm B_1$ (and superharmonic on $B_{1}^C$). The constant $\d = \d(\eta) > 0$ will be chosen below. Then for a fixed $s$ and $t < s$, define $T^2(a) := \frac{1}{(1 + \d)^2}g^2(s) - 2 (1 + \d)^2 h'(1) (s - a) m(s)$ and $w(t, x) := m(s) h((1 + \d)|x|/T(-t))$; here and in what follows $h'(1)$ means $h'(1+)$. 
	
	The function $T(-\tau)$ is positive whenever
\[
    -s\leq \tau\leq 0,
    \qquad \tau<-s+L_s,
    \text{ where }
    L_s:=\frac{g^2(s)}
    {2(1+\delta)^4m(s)h'(1)}.
\]
	For future reference, $T'(a) = (1 + \d)^2 T^{-1} h'(1) m(s) > 0$.
	
	We claim that $w$ is a weak supersolution to the Stefan problem (see \cite{KP11}). The function $w$ is piecewise smooth with derivative jumps across the curves $|x| = T(-t)$, $|x| = \frac{1}{1 + \d}T(-t)$, so it suffices to check this in the classical sense. If $\frac{1}{1 + \d} T(-t) < |x|$, we have that $\Delta w \leq 0$ (as a distribution) and $w_t = (1 + \d)|x| m(s) h'((1+\d)|x|/T(-t)) T^{-2}(-t)T'(-t) \geq 0$, so $w_t - \Delta w \geq 0$. Along $|x| = \frac{1}{1+\d}T(-t)$, we instead need to check the Stefan condition $w_t \geq |\nabla w|^2$ (from the $|x| > \frac{1}{1+\d}T(-t)$ side). There we have $|\nabla w|^2 = w_t = (1 + \d)^2 m^2(s) (h'(1))^2 T^{-2}(-t)$ by construction.

	Along $t = -s$, we have that $T(-t) = \frac{1}{1 + \d} g(-t)$, so for $|x| < T(-t)$ we know that $u(t, x) \leq 0 \leq w(t, x)$. On the complement, we instead have $w(t, x) = m(s) \geq u(t, x)$. Along $\partial B_{2 g(s)}$, we have that $w(t, x) = m(s)$ at $t = -s$ as we have already checked, and $w_t \geq 0$ (this holds globally), so $w(t, x) = m(s) \geq u(t, x)$. We conclude from the weak comparison principle in \cite[Proof of Lemma 4.1]{KP11} that 
	for every $a\in(0,s)$ satisfying $s-a<L_s$,
\[
    u\leq w
    \quad\text{on }[-s,-a]\times B_{2g(s)}.
\]
	Fix $a\in(0,s)$. We claim that
\[
    g^2(s)
    \leq (1+\delta)^4g^2(a)
       +2(1+\delta)^4h'(1)(s-a)m(s).
\]

If $s-a\geq L_s$, then
\[
    2(1+\delta)^4h'(1)(s-a)m(s)\geq g^2(s),
\]
so the claim follows directly. This includes the collapse
endpoint $s-a=L_s$.

Suppose next that $s-a<L_s$. If $g(a)\geq 2g(s)$,
then the claim is immediate from
\[
    g^2(s)\leq \frac14g^2(a)\leq(1+\delta)^4g^2(a).
\]
It therefore remains to consider $g(a)<2g(s)$.

Choose $0<\kappa<2g(s)-g(a)$ and a point $x$ with
$|x|=g(a)+\kappa$. Then $x\in B_{2g(s)}$ and
$u(-a,x)>0$. By comparison, $w(-a,x)>0$, and hence
\[
    1<(1+\delta)\frac{g(a)+\kappa}{T(a)}.
\]
Letting $\kappa\searrow0$ gives
\[
    T(a)\leq(1+\delta)g(a).
\]
Squaring and substituting the definition of $T^2(a)$ yields
\[
    g^2(s)
    \leq(1+\delta)^4g^2(a)
       +2(1+\delta)^4h'(1)(s-a)m(s),
\]
as claimed.

Choose $\delta=(1+\eta)^{1/4}-1$ and set
\[
    C_\eta:=2(1+\delta)^4h'(1).
\]
Then $C_\eta$ depends only on $n$ and $\eta$, and the
claimed estimate follows.

Letting $a\searrow0$ and using $g(0)=0$, we obtain
\[
    \frac{g^2(s)}{s}\leq C_\eta m(s).
\]
Since $g(s)\to0$ and $u$ is continuous at $(0,0)$
with $u(0,0)=0$, we have $m(s)\to0$. Therefore
\[
    \lim_{s\searrow0}\frac{g(s)}{\sqrt{s}}=0.
\]
\end{proof}

\begin{remark}
	Notice that this proof does not really use the radial structure of our setting; it works for any viscosity solution to the Stefan problem with a melting singularity and compact negative phase. Integrating the estimate in Lemma \ref{lem:g2primesmall} gives an alternative proof of this slower-than-parabolic melting, but uses the extra structure more heavily. We retain the previous argument to emphasize that this result is not original to us and can be expected in essentially all situations.
\end{remark}

Now we are in a position to show that $g' \geq 0$. The idea is elementary: in the one-phase problem with $u\geq 0$, this is automatic from the Stefan condition, which reads $g'(t) = |\nabla u_+| \geq 0$. In the two-phase case, as the negative phase is constricted to a very small set relative to the parabolic scaling, $u_-$ (and its derivatives) will decay rapidly to $0$ as $t\rightarrow 0$. The other phase, $\{u > 0\}$, is large, and here we expect that $|\nabla u_+| > c > 0$ (and in fact blows up) at the free boundary as $t \rightarrow 0$, making the same argument work using the Stefan relation $g'(t) = |\nabla u_+| - |\nabla u_-|$. To make this rigorous, we need to be able to use a Hopf-type estimate on the positive phase to control the derivative of $u_+$ from below. Lemma \ref{lem:KP} is not quite enough for this, so we begin with a further local estimate.

\begin{lemma}\label{lem:curvature}
	Let $v$ be a classical solution of the Stefan problem on a cylinder $Q_* = [-R^2, 0] \times B_R$. Then
	\[
		\sup_{[-R^2/4, 0]\times B_{R/2}} |\nabla v| \leq C\left(\frac{1}{R} + \sup_{\partial\{v > 0\} \cap Q_*}|H|\right) \sup_{Q_*} |v|,
	\]
	where the constant $C$ depends only on $n$, and $H$ the mean curvature of the interface $\partial \{v > 0\}$, oriented so that it is positive for convex $\{ v<0 \}$ (indeed, for our radial solution $H = \frac{n - 1}{g(-t)}$).
\end{lemma}

\begin{proof}
	We begin by noting that this estimate is invariant under the $0$-homogeneous parabolic rescaling $v(R^2 t, R x)$, so we may assume that $R = 1$. Set $m := \sup_{Q_*} |v|$. If $m=0$, the estimate is trivial, so let from now on $m>0$. 

	The proof proceeds via a Bernstein estimate for 
	\[
		F = (1 + t) \eta^2(x) \phi(v) |\nabla v|^2.
	\]
	Here $\eta$ is a smooth cutoff function which is $1$ on $B_{1/2}$ and $0$ on $\partial B_1$, while
	\[
		\phi(v) = \left(1 - \frac{|v|}{2 m}\right)^{-\frac{1}{2}}.
	\]

	As $F$ vanishes on the parabolic boundary of $Q_*$, its supremum must either be attained at a point $(\bar{t}, \bar{x})$ in $Q_* \cap \{v \neq 0\}$, or along a sequence tending towards a point on $Q_* \cap \{v = 0\}$ (note that $F$ is not continuous across $\{v = 0\}$). We analyze the former case first.

	Here we can compute the heat operator applied to $F$, using $L = \partial_t - \Delta$ and starting by differentiating the factors:
	\[
		L |\nabla v|^2 = -2 |D^2 v|^2\leq 0 \qquad L \phi(v) = \phi'(v) Lv - \phi''(v) |\nabla v|^2 = - \phi''(v)|\nabla v|^2.
	\]
	We also have that at this point, $\nabla F = 0$, so
	\[
		\nabla |\nabla v|^2 = - |\nabla v|^2 \left(\frac{2\nabla \eta}{\eta} + \frac{\phi'(v)}{\phi(v)}\nabla v\right).
	\]
	Using this in some cross-terms, we get
	\begin{align*}
		0 &\leq LF \\
		&\leq \phi(v)|\nabla v|^2 \left(\eta^2 + (1 + t)\bigg[ -\Delta \eta^2 - \eta^2 \frac{\phi''(v)}{\phi(v)}|\nabla v|^2 \right. \\
		&\qquad \left. + 2 \left(\frac{2\nabla \eta}{\eta} + \frac{\phi'(v)}{\phi(v)} \nabla v\right)\cdot \left(\eta^2 \frac{\phi'(v)}{\phi(v)}\nabla v + 2 \eta \nabla \eta\right) - 4 \eta \nabla \eta \cdot \nabla v  \frac{\phi'(v)}{\phi(v)}     \bigg]  \right) \\
		& \leq \phi(v)|\nabla v|^2 \left(\eta^2 + |\Delta \eta^2| + (1 + t)|\nabla v|^2 \eta^2 \left[(2 + \e) \left(\frac{\phi'(v)}{\phi(v)}\right)^2 - \frac{\phi''(v)}{\phi(v)}\right]  +  C_\e |\nabla \eta|^2   \right),
	\end{align*}
	using $1 + t \leq 1$, and using Cauchy-Schwarz on the cross-terms involving $\phi'(v)/\phi(v)$. Now, by direct computation
	\[
		\phi' = \frac{\pm 1}{4m}\phi^3 \qquad \phi'' = \frac{3}{16 m^2} \phi^5,
	\]
	with $\pm$ depending on the sign of $v$. In either case,
	\[
		\frac{\phi''}{\phi} -  (2 + \e)\left(\frac{\phi'}{\phi}\right)^2 = \frac{1 - \e}{16m^2} \phi^4,
	\]
	so for sufficiently small $\e$ the corresponding term above is strictly negative. We arrive at
	\[
	0 \leq LF \leq \phi(v)|\nabla v|^2 \left[ C(\eta) - (1 + t)|\nabla v|^2 \eta^2 \phi^4(v)  \frac{1}{32m^2} \right],
	\]
	or
	\[
		F \leq C(\eta) m^2 \frac{1}{\phi^3(v)} \leq C(\eta) m^2.
	\]
	Restricting to $[-1/4, 0] \times B_{1/2}$, this immediately implies the conclusion of the lemma.

	We now turn to the case of $F$ being maximized along $\{v = 0\}$ instead. While $F$ is discontinuous along $\{v = 0\}$, it is piecewise smooth on each side, with a jump discontinuity, so the maximum will be attained at some point $(\bar{t}, \bar{x})$ from one of the sides. Let $e_n$ be the outward unit normal to $\{v > 0\}$ at this point, and express $\{v = 0\}$ as the graph $x_n = g(t, x')$ locally, with $\nabla_{x'} g(\bar{t}, \bar{x}') = 0$ and $g_t(\bar{t}, \bar{x}')$ being the normal velocity. It will be helpful to use the notation $f_\pm$ to denote the limit values from $\{v > 0\}$ (resp. $\{v < 0\}$) of piecewise continuous functions, and likewise for derivatives.

	We compute the normal derivatives of $F_{\pm}$ at $(\bar{t}, \bar{x})$. Differentiating the relation $v_\pm(t, x', g(t, x')) = 0$ once gives the standard relation $g_t (v_\pm)_n = - (v_\pm)_t$ for the normal velocity (everything is evaluated at $(\bar{t}, \bar{x})$). Differentiating twice in space gives that $\Delta v_\pm - (v_\pm)_{nn} + (v_\pm)_n H = 0$, where $H$ is the mean curvature (i.e. $ \Delta_{\bar x'}g(\bar{t}, \bar{x})$ with our sign convention, though this will not be important). As $L v_\pm = 0$, $ \Delta v_\pm = (v_\pm)_t = - g_t (v_\pm)_n$, or
	\[
		(v_\pm)_{nn} = (H - g_t) (v_\pm)_n.
	\]
	As $\partial_n |\nabla v_\pm|^2 = 2 (v_\pm)_n (v_\pm)_{nn}$ at this point,
	\[
		\partial_n |\nabla v_\pm|^2 = 2 (v_\pm)_n^2 (H - g_t) = 2 |\nabla v_\pm|^2 (H - g_t).
	\]
	The Stefan condition gives that $g_t = |\nabla v_+| - |\nabla v_-|$, so we can rewrite as
	\[
		\partial_n |\nabla v_\pm|^2 = - 2 |\nabla v_\pm|^2 (|(v_+)_n| - |(v_-)_n| - H).
	\]
	We also compute that
	\[
		\phi(v_\pm) = 1 \qquad \phi(v_\pm)_n = \phi'(v_\pm) (v_\pm)_n = \frac{\pm 1}{4m} (v_\pm)_n.
	\]
	Combining everything,
	\[
		(F_{\pm})_n = F_\pm\left( \frac{2 \eta_n}{\eta} + \frac{\pm 1}{4m} (v_\pm)_n - 2(|(v_+)_n| - |(v_-)_n| - H)\right).
	\]

	First consider the case that $|\nabla v_+ (\bar{t}, \bar{x})| \geq |\nabla v_- (\bar{t}, \bar{x})|$, in which case the maximum is attained from the $\{v > 0\}$ side and $(F_+)_n \geq 0$. We then have
	\[
		0 \leq \frac{2 \eta_n}{\eta} + \frac{ 1}{4m} (v_+)_n + 2 H,
	\]
	so
	\[
		\eta |\nabla v_+| = - \eta (v_+)_n \leq 8 (\eta |H| + |\nabla \eta|) m.
	\]
	Similarly, if the gradients are ordered the other way around, $(F_-)_n \leq 0$ and
	\[
		0 \geq \frac{2 \eta_n}{\eta} - \frac{ 1}{4m} (v_-)_n + 2 H,
	\]
	leading to
	\[
		\eta |\nabla v_-| = - \eta (v_-)_n \leq 8 (\eta |H| + |\nabla \eta|) m.
	\]
	Either way, we get
	\[
		\eta |\nabla v_\pm| \leq 8 (\eta |H| + |\nabla \eta|) m,
	\]
	and so
	\[
		\max F \leq F(\bar{t}, \bar{x}) \leq C(1 + \sup_{\partial\{v > 0\} \cap Q_*} |H|)^2 m^2.
	\]
	This again gives the conclusion of the lemma when restricted to $[-1/4, 0] \times B_{1/2}$.
\end{proof}

\begin{remark}
	Estimates like the above are common for the 1D Stefan problem, where there is no curvature contribution; see for example even the very early reference \cite{CDH} (which relies on a smallness assumption), or the more modern form in \cite{AK}. A global estimate for radial solutions in the same spirit can be found in \cite{LJ}; the approach taken above is a more standard local Bernstein estimate from quasilinear elliptic and parabolic theory.
\end{remark}

This lemma can be applied to our setting to give the following:

\begin{lemma}\label{lem:g2primesmall}
	\[
		\lim_{t \searrow 0} g(t) g'(t) = 0.
	\]
\end{lemma}

\begin{proof}
	Apply Lemma \ref{lem:curvature} to $u$ on the cylinder $[-t - g^2(t)/4, -t] \times B_{g(t)/2}(g(t))$. On this cylinder, $|x| \geq g(t)/2$, so $|H| \leq (n-1)\frac{2}{g(t)}$. We obtain that
	\[
		\sup_{[-t - g^2(t)/16, -t] \times B_{g(t)/4}(g(t))} |\nabla u| \leq C \frac{1}{g(t)} \max_{[-t - g^2(t)/4, -t] \times B_{g(t)/2}(g(t))} |u|.
	\]
	Using the Stefan condition
	\[
		g'(t) = |\nabla u_+(-t, g(t))| - |\nabla u_-(-t, g(t))|,
	\]
	we immediately get
	\[
		|g(t)g'(t)|\leq C \max_{[-t - g^2(t)/4, -t] \times B_{g(t)/2}(g(t))} |u| \rightarrow 0,
	\]
	from the continuity of $u$ and the fact that $g(t) \rightarrow 0$.
\end{proof}

With this, we are in a position to bound $|\nabla u_+|$ from below.

\begin{lemma}\label{lem:strictlymonotone}
	For $t$ small,
	\[
		|\nabla u_+(t, g(-t))| \geq c g^{-1/3}(-t) > 0.
	\]
	The constant $c$ depends on $u$ but is independent of $t$.
\end{lemma}

\begin{proof}
	We will use the comparison principle with a subsolution to the heat equation on $\{|x| > g(-t)\}$. Set
	\[
		w(t, x) = \d \frac{|x|^2 - g^2(-t)}{(|x|^2 + g^2(-t))^\a}
	\]
	with $\d$ small to be chosen shortly, and $\a \in (0, 1)$. Then $w = 0$ exactly on the interface $\{|x| = g(-t)\}$, and is positive on the rest of $\{|x| > g(-t)\}$. We differentiate $w$ to check the subsolution property (viewing it as a function of two variables $w(t, r)$ in polar coordinates):
	\begin{align*}
		w_t(t, r) &= \frac{2 \d g(-t) g'(-t)}{(r^2 + g^2(-t))^{\a + 1}}[(1 + \a) r^2 + (1 - \a) g^2(-t) ] \\
		w_r(t, r) &= \frac{2 \d r}{(r^2 + g^2(-t))^{\a + 1}}[(1 - \a) r^2 + (1 + \a) g^2(-t) ] \\
		w_{rr}(t, r)&= \frac{2 \d }{(r^2 + g^2(-t))^{\a + 2}}[(1-\a)(1 - 2 \a)r^4 + 2(1-3\a-\a^2) r^2 g^2(-t) + (1 + \a) g^4(-t)].
	\end{align*}
	Computing the Laplacian, setting $\a = 2/3$, and using $n \geq 2$,
	\begin{align*}
		w_{rr} + \frac{n-1}{r} w_r &= \frac{2 \d}{(r^2 + g^2(-t))^{\a + 2}}[(1 + \a) n g^4(-t) +(n-2\a)(1 - \a)  r^4 + 2(n - 3\a - \a^2) r^2 g^2(-t) ] \\
		&\geq \frac{2\d}{(r^2 + g^2(-t))^{\a + 2}}[10/3 g^4(-t) + 2/9 r^4 - 8/9 r^2 g^2(-t) ] \\
		&\geq \frac{2}{9}\frac{\d}{(r^2 + g^2(-t))^{\a}},
	\end{align*}
	where we used 	that
	\[10/3 g^4(-t) + 2/9 r^4 - 8/9 r^2 g^2(-t) - \frac{1}{9} (r^2+g^2(-t))^2 = \frac{1}{9} \left((r^2 - 5 g^2(-t))^2 + 4 g^4(-t) \right)\geq 0\]
	The full heat operator, then, can be estimated as
	\[
		w_t - w_{rr} - \frac{n - 1}{r} w_r \leq \frac{2\d}{(r^2 + g^2(-t))^{\a}}\left[- \frac{1}{9} + \frac{(1 + \a) r^2 + (1 -\a) g^2(-t)}{r^2 + g^2(-t)} g(-t) g'(-t)\right] < 0
	\]
	for $t$ sufficiently small, using Lemma \ref{lem:g2primesmall} and the fact that $\frac{(1 + \a) r^2 + (1 - \a) g^2(-t)}{r^2 + g^2(-t)} \leq 1 + \a = 5/3$.

	Now choose $\d$ small enough that $w \leq u$ at $t = T_0$ small but fixed and $|x| \geq g(-T_0)$. Indeed, this is always possible with a $\d = \d(u)$, using that $u > 0$, $g$ is smooth on a neighborhood of $(T_0, g(-T_0))$, and the parabolic Hopf lemma and Harnack inequality: $u(T_0, x) \geq c(|x| - g(-T_0))$ there. Similarly, we may ensure that for $\d$ small, $w \leq u$ along a lateral boundary $|x| = R$: there $w(t, R) \leq \d  R^{2 - 2\a}$ when $g(-t) \ll R$, while $u \geq c$. Applying the comparison principle leads to $w \leq u$ on $\{(t, x) : t \in [T_0, 0], g(-t) < |x| < R\}$.

	From the computation above, $w_r(t, g(-t)) = \frac{\d 4 g^3(-t)}{(2 g^2(-t))^{1 + \a}} = \d 2^{1 - \a} g^{1 - 2 \a}(-t) \geq c g^{-1/3}(-t)$, which then immediately gives the conclusion.
\end{proof}

A closer examination of the barrier construction shows that there is an $\a(n)\leq 1$, given as a root to a polynomial, such that for all $\a \in (1/2, \a(n))$ one may conclude that $|\nabla u_+ (t, g(-t))| \geq c g^{1 - 2 \a}(-t) $. It is easy to see that $\a$ increases with $n$, and computation shows that $\a(2), \a(3) < 1$ while $\a(4) = 1$. This will not be relevant below, as this proof does not give a sharp relation even when $\a(n) = 1$: the estimate $|\nabla u_+ (t, g(-t))| \geq c \frac{1}{g(-t)}$ is false in all dimensions. For our purposes, it suffices that $|\nabla u_+ (t, g(-t))| \rightarrow \infty$ as $t \rightarrow 0$.

We are now ready for the main theorem of this section, which asserts that $g$ is strictly monotone.

\begin{theorem}
	For $t$ small, $g'(t) > c g^{-1/3}(t) > 0$. The constant $c$ depends on $u$, but is uniform in $t$.
\end{theorem}

\begin{proof}
	First, set
	\[
		m(t) := \max_{0 < s \leq t} |g(s) g'(s)| \rightarrow 0.
	\]
	We integrate $gg'$ to get
	\begin{equation}\label{eq:locallip}
		|g^2(t) - g^2(s)| \leq 2 m(t)(t - s)
	\end{equation}
	for $s < t$.

	We estimate $|u_-|$, using only that it solves the heat equation on a thin domain. Let $\phi$ be the first eigenfunction of $-\Delta$ on $B_2$ with Dirichlet data $0$, normalized so that $\phi \geq 1$ on $B_1$ (and $\phi \leq C(n)$). We use
	\[
		w(t, x) = \max |u_-(T, \cdot)| \frac{g^N(-t)}{g^N(-T)} \phi\left(\frac{x}{g(-t)}\right)
	\]
	in a comparison principle for the heat equation with $|u_-|$ on $\{t \in [T, 0] : |x| < g(-t)\}$. Indeed, clearly $|u_-| = 0 < w$ on the lateral boundary of this region, while at time $T$ we have that $w \geq \max |u_-(T, \cdot)| \geq |u_-|$ on $\{ |x| < g(-T)\}$. The main thing to check is the heat equation for $w$. Setting $z := x/(g(-t), g := g(-t), g' := g'(-t)$,
	\begin{align*}
		w_t -\Delta w &= \frac{\max |u_-(T, \cdot)|}{g^N(-T)} \phi(z) g^{N-2} \left[-N g g' +  gg'\frac{z\cdot \nabla \phi(z)}{\phi(z)}+ \lambda_1(B_2) \right] \\
		&\geq \frac{\max |u_-(T, \cdot)|}{g^N(-T)} \phi(z) g^{N-2} \left[\lambda_1(B_2) - (N + \sup_{B_1} |\nabla \phi|) |g g'| \right]\\
		&\geq \frac{\max |u_-(T, \cdot)|}{g^N(-T)} \phi(z) g^{N-2} \left[\lambda_1(B_2) - (N + \sup_{B_1} |\nabla \phi|) m(-T) \right].
	\end{align*}
	Fix $N$ large and choose $T=T(u, N)$ sufficiently close to zero such that the term in brackets is strictly positive. Then $w_t - \Delta w \geq 0$, and the comparison principle gives
	\begin{equation} \label{eq:negativeexpdecay}
		|u_-(-t, x)| \leq w(-t, x) \leq \max |u_-(T, \cdot)| \frac{g^N(t)}{g^N(-T)} \max \phi \leq C_N g^N(t) \text{ for small } t>0.
	\end{equation}

	We can turn this estimate on $|u_-|$ into an estimate on $|\nabla u_-|$ along $\{u = 0\}$, with the help of our local Lipschitz bound on $g$. Fix a time $s$, and let
	\[
		v(t, r) = u(s + q^2 t, g(-s - q^2 t) + q r).
	\]
	This is a change of variables which flattens the set $\{u < 0\}$ to $\{r < 0\}$ locally, with a scaling $q$ to be chosen; here we view $u$ as a function of two variables $(t, r)$. We may compute a PDE for $v$: as
	\[
		u_t(t, r) - u_{rr}(t, r) - \frac{n - 1}{r} u_r(t, r) = 0,
	\]
	we get
	\[
		v_t(t, r) - v_{rr}(t, r) + v_r(t, r) \left[q g'(-s - q^2 t) - \frac{n - 1}{ r + g(-s - q^2 t)/q}\right] = 0.
	\]
	Choose $q = g(-s)/4$ and perform this for $t \in [-1, 0]$, $r \in [-1, 0]$. Then
	\[
		g^2(-s - q^2 t)/q^2 = g^2(-s)/q^2  - 2 \int_0^t g(-s - q^2 z) g'(-s - q^2 z) dz,
	\]
	giving
	\[
		|g^2(-s - q^2 t)/q^2 - 16| \leq 2 m(-s + q^2) \leq \e
	\]
	if $s$ is small enough. In particular we get, say, $3 \leq g(-s - q^2 t)/q \leq 5$, and so
	\[
		q |g'(-s - q^2 t)| \leq \frac{\e q}{g(-s - q^2 t)} \leq 1.
	\]
	This means that over the region we are considering, the coefficient in front of $v_r$ is bounded by $C = C(n)$. From Lemma \ref{lem_append} in the Appendix, we get that
	\begin{align*}
 \sup_{-1/4\leq t\leq0}|v_r(t,0-)|
 &\leq C(n)\sup_{[-1,0]\times[-1,0]}|v|
 \leq C(n)
 \sup_{\substack{\tau\in[s-q^2,s]\\|x|\leq g(-\tau)}}|u_-(\tau,x)|\\
 &\leq C_N\sup_{\tau\in[s-q^2,s]}g^N(-\tau)
 \leq C_N' g^N(-s),
\end{align*}
	where $C_N=C_N(u)$.
	The second-to-last step used \eqref{eq:negativeexpdecay}, while the last step follows from $g(-\tau)\leq 5q = (5/4) g(-s)$. Changing variables back to $u$, we conclude that for fixed $N$ and $s$ sufficiently small,
	\begin{equation}\label{eq:fromabove}
		|\nabla u_-(s, g(-s))| \leq \frac{1}{q} |v_r(0,0-)| \leq C_N' g^{N-1}(-s).
	\end{equation}
	Put together \eqref{eq:fromabove} and the conclusion of Lemma \ref{lem:strictlymonotone}. Via the Stefan condition, this gives
	\[
		g'(-s) = |\nabla u_+(s, g(-s))| - |\nabla u_-(s, g(-s))| \geq c g^{-1/3}(-s) - C_N' g^{N-1}(-s) \geq \frac{c}{2} g^{-1/3}(-s)
	\]
	for small $s$.
\end{proof}

From now on, we will also assume that on all of $(0, 1)$,
\begin{equation}\label{eq:a0}
	\frac{g(t)}{\sqrt{t}} \leq A_0 \text{ and } g' \geq 0,
\end{equation}
where the constant $A_0$ can be chosen small in terms of universal dimensional parameters only. This is not truly necessary but simplifies calculations in some places and can always be arranged after one additional degree-0 homogeneous rescaling. Note that under such a rescaling $u_r(t, x) = u(r^2 t, r x)$, the new free boundary curve is exactly $g_r(t) = g(r^2 t)/r$, so if $g(t) \leq \e \sqrt{t}$ on $[0, r^2]$, then $g_r(t) \leq \e \sqrt{t}$ on $[0, 1]$.

\begin{remark}
	Integrating the conclusion $g'(t) \geq c g^{-1/3}(t)$ gives the `non-degeneracy'
	\begin{equation}\label{eq:strictlymonotone}
		\lim_{t \searrow 0} \frac{g(t)}{t} = \infty,
	\end{equation}
	which we will use in later sections to rule out extremely slow melting. Many works on melting singularities, starting with \cite{HV}, work on the ``undercooled" variant of the Stefan problem, and construct additional solutions which do not satisfy the classic Stefan problem we work with and which violate this estimate. In that setting, the barrier argument in Lemma \ref{lem:strictlymonotone} is not available, but it can still be shown that $g' \geq - C_N g^N$ using similar arguments. The analysis in the next sections remains possible, if more technical, under this condition, but we do not pursue this topic further in this paper. 
\end{remark}

\subsection{Monotonicity of the interface implies doubling}

Our next goal is to establish some a priori control over $g'$ in the ``natural" scaling for the problem, which will let us avoid making geometric assumptions on $g$ or $u$. The following is a differential doubling bound on $g$:
\begin{lemma}\label{lem:doubling}
	\[
		g'(t) \leq C \frac{g(t)}{t}.
	\]
	The constant $C$ depends on the solution $u$.
\end{lemma}

\begin{proof}
	From $g'(-t) \geq 0$, we have that $u^\pm_t(t, g(-t)) \geq 0$. We also have that $|u_t| \leq C(u)$ on the parabolic boundary of $A = [-9/10, 0)\times B_S$, with with $S$ chosen so that
\[
    2\max_{s\in[0,1]}g(s)<S<R \text{ (the size of the fixed domain on which }u \text{ is defined).}
\]  Apply the maximum principle to $u_t$ on $A \cap \{u > 0\}$ and $A \cap \{u < 0\}$, separately, to get that
	\begin{equation}\label{eq:utboundbelow}
		\Delta u = u_t \geq -C.
	\end{equation}

	Take the primitive of $u$ in $t$ (this is analogous to the Duvaut transformation common when studying the one-phase Stefan problem, but in the two-phase case does not solve an obstacle problem):
	\[
		U(t, x) = \int_{t}^0 u(s, x) ds.
	\]
	This satisfies in the sense of distributions
	\[
		\Delta U(t, x) = \int_{t}^0 \Delta u(s, x) ds = \int_t^0 \theta(u)_t(s, x) ds = \theta(u)(0, x) - \theta(u)(t, x).
	\]
	In particular, the right-hand side is controlled by $1 + 2 \Vert u\Vert_{L^\infty}\leq C(u)$, so
	\[
		|\Delta U(t, x)|\leq C(u)
	\]
	uniformly in $t, x$. Using basic elliptic estimates and the fact that $U$ is radial, 
	\[
		U_r(t,r) = r^{1-n} \int_0^r s^{n-1} \Delta U(t,s) \> ds, \quad |U(t, x) - U(t, 0)| \leq C |x|^2.
	\]
	As $u(t, 0) \leq 0$, we have $U(t, 0) \leq 0$, so $U(t, x) \leq C |x|^2$. However, we also can estimate $U$ using \eqref{eq:utboundbelow}:
	\[
		U(t, x) = \int_{t}^0 u(s, x) ds \geq \int_t^0 u(t, x) + \int_t^s u_t(z, x) dz ds \geq |t| u(t, x) - C t^2.
	\]
	Combining gives
	\[
		u(-t, x) \leq C\left(t + \frac{|x|^2}{t}\right).
	\]

	Next, we turn this into a bound on $|\nabla u_+(-t, g(t))|$. Indeed, viewing $u$ as a function of two variables $(t, r)$, we have from \eqref{eq:utboundbelow} that
	\[
		r^{1-n} (r^{n-1}u_r)_r = u_{rr}(-t, r) + \frac{n - 1}{r} u_r(-t, r) \geq - C.
	\]
	Integrating in $r$ from $g(t)$ to $z \in [g(t), 2g(t)]$ leads to
	\[
		u_r(-t, z) \geq 2^{1-n} u_r(-t, g(t)+) - C g(t).
	\]
	Integrating another time in $z$, from $g(t)$ to $2g(t)$, gives
	\[
		u(-t, 2 g(t)) \geq u(-t, g(t)) + [2^{1-n} u_r(-t, g(t)+) - C g(t)]g(t) \geq 2^{1-n} u_r(-t, g(t)) g(t) - C g^2(t).
	\]
	Recalling that in this notation $|\nabla u_+(-t, g(t))| = u_r(-t, g(t)+)$,
	\[
		|\nabla u_+(-t, g(t))| \leq C \left[g(t) + \frac{u(-t, 2 g(t))}{g(t)}\right] \leq C \left[g(t) + \frac{t}{g(t)} + \frac{g(t)}{t}\right] \leq C \left[\frac{t}{g(t)} + \frac{g(t)}{t}\right].
	\]

	We can use this to estimate $g'$ from above using the Stefan condition:
	\[
		g'(t) = |\nabla u_+(-t, g(t))| - |\nabla u_-(-t, g(t))| \leq C \left[\frac{t}{g(t)} + \frac{g(t)}{t}\right].
	\]
	The conclusion follows after using \eqref{eq:strictlymonotone}.
\end{proof}

For future reference, we also record an integrated form of this estimate:
\[
	z^{C}[z^{-C} g^2(z)]'  = [g^2(z)]' - C \frac{g^2}{z} \leq 0.
\]
Integrate on $z \in [t, s]$ to give
\[
	g^2(s) \leq \left(\frac{s}{t}\right)^C g^2(t).
\]
If $s \leq 2 t$, this becomes
\begin{equation}\label{eq:intdoubling}
	g^2(s) \leq C g^2(t).
\end{equation}

\newcommand{\Ip}{I^{\text{after}}}
\newcommand{\Itrans}{I^{\text{trans}}}
\renewcommand{\Im}{I^{\text{before}}}

\section{The potential representation}\label{sec:rep}

In this section we introduce the core premise of the arguments and perform some basic analysis which will be useful later. We choose the normalized initial time strictly inside the original
interval of classical existence, so that $u$ is classical for
$-1\leq t<0$ and the conclusions of the preceding section hold
on the normalized interval. We begin with the potential representation defining
\[
	K(t, x) := \int_t^1\int_{|y| = g(s)} G(s - t, y - x)  g'(s) dS(y) ds \in [0, +\infty]. 
\]
This quantity should be interpreted as follows: as $u$ is a solution of the Stefan problem, it formally satisfies the heat equation with a right-hand side: $u_t - \Delta u = [u - \theta(u)]_t$. The function $-K(-t, x)$ is a solution to the Cauchy problem for the same PDE (i.e. on $(-1, 0] \times \R^n$)  with initial condition $0$. We should expect that, so long as $K(t, x)$ is well-defined and finite, 
and $u(t,x)+K(-t,x)$ should be caloric.

\begin{lemma}
	The integral defining $K(-t, x)$ is absolutely convergent for all $(t, x) \in Q_R$. Moreover, $u$ admits the representation formula
	\[
		u(t, x) = \ufree(t, x) - K(-t, x), 
	\]
	where $\ufree$ solves the boundary value problem
	\[
	\begin{cases}
		(\ufree)_t - \Delta \ufree = 0 & \text{ on } (-1, 0] \times B_R, \\
		\ufree(-1, x) = u(-1, x), & x \in B_R, \\
		\ufree(t, x) = u(t, x) + K(-t, x), & (t, x) \in [-1, 0) \times \partial B_R.
	\end{cases}
	\]
\end{lemma}

\begin{proof}
	So long as $|x| \neq g(-t)$, it follows from the fact that $G$ is smooth away from $(0, 0)$ 
	as well as
	$$\int_0^1 \int_{|y| = g(s)} g'(s) dS(y) \> ds = n |B_1| \int_0^1  g(s)^{n-1} g'(s) \> ds = |B_1| g^n(1) < \infty$$
	that $K(-t, x)$ is absolutely convergent and smooth. 
	By assumption (2), for $x\in \partial B_R$, $|x-y| \geq R/2$, implying that
	$$\sup_{-1\leq t \leq 0, |x|=R} K(-t,x) \leq C(R,n) g^n(1).$$
	By dominated convergence we obtain continuity at the lateral boundary.
	It is also routine to check that for $|x| = g(-t)$ and $t < 0$, $K(-t, x)$ remains absolutely convergent and continuous: on a small neighborhood of this point we have $|g'| \leq C$ and $c \leq g(-t)/g(s) \leq C$, so integrating the exponential explicitly in cylindrical coordinates gives
	\begin{align*}
		\int_{|y| = g(s)} G(s + t, y - x) dS(y) &\leq C g^{n-1}(s)(s + t)^{-n/2} \int_{S^{n-1}}  e^{- c(1 - y_n)/(s + t)} dS(y) \\
		&\leq C (s + t)^{-n/2}\int_{-1}^1 e^{-c(1 - y_n)/(s + t)} (1 - y_n^2)^{\frac{n-3}{2}} dy_n \\
		&\leq C (s + t)^{\frac{- 1}{2}}.
	\end{align*}
	See the proof of Lemma \ref{lem:transfromabove} below for an explicit computation of this integral. Then estimating $K$ near the pole,
	\[
		\int_{-t}^{-t+\e}\int_{|y| = g(s)} G(s + t, y - x)  g'(s) dS(y) ds \leq C \int_{-t}^{-t + \e} (s + t)^{-1/2} ds \leq C\sqrt{\e}.
	\]
	The constant in this estimate can be chosen uniformly
for $(t,x)$ near any fixed interface point with
$-1<t<0$. Splitting off this uniformly small contribution
and applying dominated convergence to the remaining
integral proves continuity of $K(-t,x)$.

	We will now justify the relation
	\[
		u(t, x) = \ufree(t, x) - K(-t, x), 
	\]
	for $t < 0$. In particular, we show that $u(t, x) + K(-t, x)$ solves the heat equation on $(-1, 0) \times B_R$ in the sense of distributions, after which this follows by comparison principle and the definition of $\ufree$. To see this, take a $\phi \in C_c^\infty((-1, 0) \times B_R)$, and obtain from the fact that $u$ is a distributional solution to the Stefan problem that
	\[
		\int \theta(u) \phi_t + u \Delta \phi dx dt = 0,
	\]
	so
	\[
		\int u \phi_t + u \Delta \phi dx dt = -\int [\theta(u) - u] \phi_t dx dt = -\int_{\{u > 0\}} \phi_t dx dt.
	\]
	On the other hand,
	\begin{align*}
		\int (\phi_t + \Delta \phi) K(-t, x) dx dt &= \int \int_{|y| = g(s)} \int_{-s}^0 \int G(s + t, y - x) g'(s) (\phi_t + \Delta \phi)(t, x) dx dt dS(y) ds \\
		&= -\int \int_{|y| = g(s)} \phi(-s, y) g'(s) dS(y) ds;\\
	\end{align*}
	extend $\phi$ by zero outside $B_R$. Fubini's theorem
applies because
\[
\begin{aligned}
&\int_0^1\int_{|y|=g(s)}g'(s)
  \int_{-s}^0\int_{\mathbb R^n}
  G(s+t,y-x)|\phi_t+\Delta\phi|(t,x)
  \,dx\,dt\,dS(y)\,ds\\
&\qquad\leq
  \|\phi_t+\Delta\phi\|_{L^\infty}
  |B_1|g^n(1)<\infty.
\end{aligned}
\]
	After changing variables, we have shown that
	\begin{align*}
		\int \int (\phi_t + \Delta \phi) (u(t, x) + K(-t, x)) dx dt &= \int \left[\int_{|x| < g(-t)} \phi_t dx - \int_{|x| = g(-t)} \phi(t, x) g'(-t) dS(x)\right] dt \\
		&= \int \partial_t \int_{|x| < g(-t)} \phi dx dt \\
		&= 0,
	\end{align*}
	as claimed. As $g'$ is bounded up to $t=1$, 
	$$ \sup_{x\in \bar B_R} K(a,x) \leq C \sqrt{1-a} \to 0, a \uparrow 1,$$ 
	the initial data is attained uniformly.

	It remains to check that $K(-t, x)$ remains absolutely convergent at $t = 0$, $x = 0$. As $u$, $\ufree$ are bounded (by the maximum principle, $\Vert \ufree\Vert_{L^\infty(Q_R)} \leq \Vert u\Vert_{L^\infty(Q_R)} + C(R,n) g^n(1)$), we have that $K(-t, x)$ is bounded uniformly in $t$ for $t < 0$. We also have that the integrands $\int_{|y| = g(s)} G(s + t, x_t + y) dS(y) 1_{\{s > -t\}}$ converge pointwise to $\int_{|y| = g(s)} G(s, y) dS(y)$ for every $s > 0$ and curve $x_t \rightarrow 0$. Then by Fatou's lemma (with respect to the measure $g'(s)ds$), $K(0, 0)$ is absolutely convergent and
	\[
		0 \leq K(0, 0) \leq \liminf K(-t, x_t) = \liminf \ufree(t, x_t) - u(t, x_t) = \ufree(0, 0) - u(0, 0).
	\]

	Finally, to see that this is actually an equality, select $x_t = 0$ and consider $K(-t, 0)$. Restrict the region of integration to $[-2t,1]$: here, we have that $(t + s)^{-n/2} \leq 2^{n/2} s^{-n/2}$ while $e^{-g^2(s)/(4(s + t))} \leq e^{-g^2(s)/(4s)}$, so the integrand
	\[
		\int_{|y| = g(s)} G(s + t, y)  g'(s) dS(y) \leq 2^{n/2} \int_{|y| = g(s)} G(s, y)  g'(s) dS(y).
	\]
	As we have just established that the right-hand quantity is integrable, these pieces converge by the dominated convergence theorem, i.e.
	\[
		\lim_{t \rightarrow 0} \int_{-2t}^1 \int_{|y| = g(s)} G(s + t, y)  g'(s) dS(y) ds = K(0, 0).
	\]
	Over the remaining interval $[-t, -2t]$ we instead use Lemma \ref{lem:doubling} and the fact that $g^2(s)/(s) \rightarrow 0$:
	\begin{align*}
		\int_{-t}^{-2t} \int_{|y| = g(s)} G(s + t, y)  g'(s) dS(y) ds &\leq C \int_{-t}^{-2t} \frac{g^{n}(s)}{s} (s + t)^{-n/2} e^{-g^2(s)/(4(s + t))} ds \\
		&\leq C \int_0^{-t} \frac{g^n(-t)}{ -t z^{n/2}} e^{-g^2(-t)/(4z)} dz \\
		&\leq C \frac{g^2(-t)}{-t} \int_0^{-t/g^2(-t)} z^{-n/2} e^{-1/(4z)} dz.
	\end{align*}
	The second-to-last step used $g(-t) \leq g(-2 t)\leq C g(-t)$, while the last was another change of variables. The integral is always convergent near $0$; in $n \geq 3$ it is bounded by a constant by integrating to $\infty$, while if $n = 2$ it is controlled by $C \log(-t/g^2(-t))$. Either way, this quantity goes to $0$ as $-t$, and hence $g^2(-t)/(-t)$, goes to $0$.
	This means that
	\begin{align*}
		K(0, 0) &= \lim_{t \rightarrow 0} \int_{-2t}^1 \int_{|y| = g(s)} G(s + t, y)  g'(s) dS(y) ds \\&= \lim_{t \rightarrow 0} \int_{-t}^1 \int_{|y| = g(s)} G(s + t, y) g'(s) dS(y) ds = \lim_{t \rightarrow 0} K(-t, 0),
	\end{align*}
	as we were trying to show.
\end{proof}
The basic premise of the rest of our approach is as follows: we evaluate the formula above at $0 = u(t, g(-t))$, a free boundary point. At such points, the contribution from $\ufree$ is straightforward to estimate: despite the dependence on $K$ in the boundary condition, this function is caloric and therefore smooth in a parabolic neighborhood of $0$, and so admits an expansion
\[
	\ufree(-t, g(t)) = \ufree(0, 0) - t (\ufree)_t(0, 0) + \frac{g(t)^2}{2} (\ufree)_{rr}(0, 0) + \ldots
\]
(odd terms in $r$ are missing because $\ufree$ is radial). This is not directly useful due to the presence of the first, zero-order term, which is generally nonzero and dominates the expression. Therefore, we instead look at
\[
	0 = u(-t, g(t)) - u(0, 0),
\]
which gives instead that
\[
	|K(t, g(t)) - K(0, 0)| \leq  C(t + g^2(t)).
\]
Looking at the integral $K(t, g(t))$, the integrand has a pole at $(t, g(t))$, and so we might hope to estimate it roughly as $g'(t)$ multiplied by some explicit function of $t$. This leads to an approximate differential relation for $g$, which can then be integrated to find its asymptotic behavior near $t = 0$.

The actual arguments below are more subtle (after all, $K(t, g(t))$ is a nonlocal quantity which cannot actually be estimated in such a way rigorously, the resulting differential inequalities are nontrivial to integrate, and there are certain anomalous behaviors for $g$ which must be ruled out separately), but we will still begin by studying the behavior of the integral $K(t, g(t))$.

\section{Potential estimates on the interface} \label{sec:potential}

Let $I_t = K(t, g(t)) - K(0, 0)$. To pin down the behavior of this quantity, we will break the (combined) integrand into three regions:
\[
	I_t = \Ip_t + \Itrans_t + \Im_t,
\]
where
\begin{align*}
	\Ip_t &= - \int_0^t \int_{|y| = g(s)} G(s, y)  g'(s) dS(y) ds, \\
	\Itrans_t &= \int_t^{t + L(t)}\int_{|y| = g(s)} [G(s - t, y - g(t)) - G(s, y)] g'(s) dS(y) ds, \\
	\Im_t &= \int_{t + L(t)}^1\int_{|y| = g(s)} [G(s - t, y - g(t)) - G(s, y)] g'(s) dS(y) ds.
\end{align*}
Here $L(t)\geq 0$ is a function to be chosen below. The basic idea is to choose $L(t)$ in a way where $\Itrans_t$ has a positive integrand (which is clearly the case if $L(t)$ is very small, as the first term has a pole at $s = t$ while the second does not) and is much larger than $\Im_t$ (where we expect heavy cancellation). This will give that $I_t$ is dominated by $\Itrans_t$, which is comparable to $g'(t)$ multiplied by an explicit function of $t$, and $\Ip_t$, which is a kind of weighted primitive of $g'$. As observed earlier, we also have that $I_t$ is overall small: $|I_t| = |\ufree(-t, g(t)) - \ufree(0, 0)| \leq C (t + g^2(t))$. Together, this will lead to differential inequalities for (some transformations of) $g$, which will be studied in the next section.

We note that the behavior is somewhat different in 2D, where the pole is less severe and so $\Im_t$ contributes substantially to the total integral.

\subsection{Estimating the integrals: transition region}

There are two goals in estimating $\Itrans_t$: (1) to estimate the quantity $\int_{|y| = g(s)} G(s - t, y - x) - G(s, y) dS(y)$ pointwise from below (and above), and (2) use this to also estimate the largest choice one can make for $L(t)$ for which the integrand remains positive. The first goal is a completely explicit geometric calculation, as this spatial integrand has only a superficial dependence on $g$.

We begin by estimating, for $X = R' e_n$ with $R' < R$ and $e_n$ an arbitrary unit vector,
\[
J(R, R', \t) := \int_{|y| = R} G(\t, y - X) dS(y) = c(n)\int_{S^{n-1}} \t^{-n/2} e^{-\frac{|RY - R'e_n|^2}{4\t}}  R^{n-1} dS(Y).
\]
\begin{lemma}\label{lem:gaussianfrombelow}
	If $\frac{R - R'}{\sqrt{\t}} \leq \d < 1$, then
	\[
		J(R, R', \t) \geq c(n) e^{-C\d} \t^{-1/2} \min\{(R/\sqrt{\t})^{n-1}, 1\}.
	\]
\end{lemma}

\begin{proof}
	Write $|RY - R'e_n|^2 = (R - R')^2 + 2 R R' (1 - Y_n) \leq C[(R - R')^2 + R^2 (1 - Y_n)]$ by expanding $Y$ into coordinates. Then
	\[
		e^{-\frac{|RY - R'e_n|^2}{4\t}} \geq e^{-C\d}e^{- C R^2 (1 - Y_n)/\t}.
	\]
	Now we compute the integral:
	\[
		\int_{S^{n-1}} e^{- C R^2 (1 - Y_n)/\t} dS(Y) \geq \int_{Y \in S^{n-1}: 1 - Y_n \leq \t/R^2 } e^{-C} \geq c \min\{(\t/R^2)^{\frac{n-1}{2}}, 1\}.
	\]
	The last inequality comes from estimating the area of a spherical cap. Plugging into the definition of $J$,
	\[
		J(R, R', \t) \geq c(n) \t^{-n/2} R^{n-1} e^{-C\d} \min\{(\t/R^2)^{\frac{n-1}{2}}, 1\} \geq c(n)e^{-C\d} \t^{-1/2} \min\{(R/\sqrt{\t})^{n-1}, 1\} .
	\]
\end{proof}

For sufficiently small $t$, uniformly for $t<s\leq2t$,
\[
 \frac{g(s)-g(t)}{\sqrt{s-t}}\leq\frac12.
\]
To see this, use \eqref{eq:intdoubling} to get
\[
	g^2(t) \leq g^2(s) \leq C g^2(t),
\]
so integrating the original estimate in Lemma \ref{lem:doubling} more carefully will give
\[
	(g(s)-g(t))^2 <2 \leq g^2(s) - g^2(t) \leq C \int_t^s \frac{g^2(z)}{z} dz \leq C (s - t) \frac{g^2(t)}{t}\leq (s - t)
\]
for small $t$ depending on $u$.

Moreover, estimating the factor in the $\min$ in the conclusion of Lemma \ref{lem:gaussianfrombelow},
\[
	\min\{(g(s)/\sqrt{s - t})^{n-1}, 1\} \geq c
\]
if $s \leq t + g^2(t)$. Integrating over this range leads to
\begin{align*}
	\int_t^{t + g^2(t)}&\int_{|y| = g(s)} G(s - t, y - x) g'(s) dS(y) ds \\
	&= \int_t^{t + g^2(t)} J(g(s), g(t), s - t) g'(s) ds \\
	& \geq c \int_t^{t + g^2(t)}  \frac{g'(s)}{\sqrt{s - t}} ds.
\end{align*}
This is the main positive factor we were looking for in $\Itrans_t$, but we still must show that the negative integral over the same interval is not an issue.

This integral can be written explicitly:
\[
	\int_t^{t + g^2(t)}\int_{|y| = g(s)} G(s, y) g'(s) dS(y) ds = c_n \int_t^{t + g^2(t)} g^{n-1}(s) s^{-n/2}e^{-g^2(s)/4s} g'(s) ds.
\]
Since $g(s)/\sqrt{s}\to0$, for sufficiently small $t$ we have
\[
    c_n\left(\frac{g(s)}{\sqrt{s}}\right)^{n-1}
    \leq\varepsilon
    \qquad\text{for every }0<s\leq2t, \text{ so that}
\]
\[
	c_n g^{n-1}(s) s^{-n/2}e^{-g^2(s)/4s} g'(s) \leq \e g'(s) s^{-1/2}\qquad\text{for every }0<s\leq2t.
\]
As $s > s - t$ this can then be pointwise absorbed into the positive integrand if $\e $ is small enough, giving
\[
\int_t^{t + g^2(t)}\int_{|y| = g(s)} [G(s - t, y - g(t)) - G(s, y)] g'(s) dS(y) ds \geq c \int_t^{t + g^2(t)}  \frac{g'(s)}{\sqrt{s - t}} ds.
\]

We summarize this computation in the following lemma:
\begin{lemma}\label{lem:transfrombelow}
	Choose $L(t) = g^2(t)$. For $t$ sufficiently small,
	\[
		\Itrans_t \geq c\int_t^{t + L(t)}  \frac{g'(s)}{\sqrt{s - t}} ds.
	\]
\end{lemma}

There is also a corresponding bound from above. It will be convenient to only give the bound on $J$ here, and then integrate in $s$ sharply later to get a more precise estimate.

\begin{lemma}\label{lem:transfromabove}
	For all $\alpha \in (0, 1)$, if $\alpha R \leq R' \leq R$, then
	\[
		J(R, R', \tau) \leq C_\alpha \min\{\t^{-1/2}, R^{n-1} \t^{-n/2}\}.
	\]
\end{lemma}

\begin{proof}
	We proceed as in Lemma \ref{lem:gaussianfrombelow}. In this case it is somewhat simpler as we can drop some of the exponential factors (indeed, $J(R, R', \t) \leq C R^{n-1} \t^{-n/2}$ is trivial). We can potentially do better by still writing $|RY - R'e_n|^2 = (R - R')^2 + 2 R R' (1 - Y_n)$ and so 
	\[
		J(R, R', \t) = c(n)\int_{S^{n-1}} \t^{-n/2} e^{-\frac{|RY - R'e_n|^2}{4\t}}  R^{n-1} dS(Y) \leq C \t^{-n/2} R^{n-1} \int_{S^{n-1}} e^{- \frac{R R' (1 - Y_n)}{2 \t}} dS(Y).
	\]
	The integral can be computed explicitly
	\begin{align*}
		\int_{S^{n-1}} e^{- a (1 - Y_n)} dS(Y) &= C\int_{-1}^1 (1 - Y_n^2)^{\frac{n - 3}{2}} e^{- a (1 - Y_n)} dY_n \\
		&= Ca^{- \frac{n - 1}{2}}\int_{0}^{2a} e^{-z} z^{\frac{n-3}{2}}(2 - z/a)^{\frac{n-3}{2}} dz \\
		&\leq C a^{- \frac{n - 1}{2}}\int_{0}^{\infty} e^{-z} z^{\frac{n-3}{2}}2^{\frac{n-3}{2}} dz \\
		& = Ca^{- \frac{n - 1}{2}}
	\end{align*}
	when $n \geq 3$. When $n = 2$ this is still valid:
	\begin{align*}
		\int_{-1}^1 (1 - Y_n^2)^{-\frac{1}{2}} e^{- a (1 - Y_n)} dY_n &= \int_{0}^\pi e^{-a(1 - \cos \theta)} d \theta \\
		 \leq \int_0^\pi e^{- a 2 \theta^2/\pi^2} d\theta \leq C a^{-1/2}.
	\end{align*}
	Plugging this in for $J$,
	\[
		J(R, R', \t) \leq C \t^{-1/2} (R/R')^{\frac{n-1}{2}}.
	\]
\end{proof}

\subsection{Estimating the integrals: the past}

We begin by observing that we chose $L(t) = g^2(t)$ in order to get the explicit lower bound in Lemma \ref{lem:transfrombelow}. However, the region over which the combined integrand $\int_{|y| = g(s)} [G(s - t, y - x) - G(s, y)] dS(y)$ remains nonnegative is somewhat larger. For the negative part we always have
\[
	c_n g^{n-1}(s) s^{-n/2}e^{-g^2(s)/4s} \leq c_n  g^{n-1}(s) s^{-n/2} (\leq \e s^{-1/2}),
\]
while for the positive part we still retain
\[
	J(g(s), g(t), s - t) \geq c (s - t)^{-1/2} \min\{(g(s)/\sqrt{s-t})^{n-1}, 1\}.
\]
If the $\min$ is $1$, then using the extra inequality in parentheses means the difference is nonnegative. Even if not, as long as $s \leq t(1 + \eta)$ with $\eta$ small, we can instead use that
\[
	s^{-n/2} \leq (s - t)^{-n/2} \eta^{n/2}.
\]
For a small fixed $\eta$, this also gives that the difference is nonnegative, and so
\begin{equation}\label{eq:pastclose}
	\int_{t + L(t)}^{t(1 + \eta)}\int_{|y| = g(s)} [G(s - t, y - g(t)) - G(s, y)] g'(s) dS(y) ds \geq 0.
\end{equation}
Note that for small $t$, $L(t) = g^2(t) \ll \eta t$, so this covers a relatively large interval. For $n = 2$ we will need to retain the remainder there as well. After reabsorbing,
\begin{equation}\label{eq:pastclose2D}
\begin{aligned}
&\int_{t + L(t)}^{t(1 + \eta)}
 \int_{|y| = g(s)}
 [G(s - t, y - g(t)) - G(s, y)] g'(s)\,dS(y)\,ds \\
&\qquad \geq c \int_{t + L(t)}^{t(1 + \eta)}
 J(g(s), g(t), s - t) g'(s)\,ds  \geq c \int_{t + L(t)}^{t(1 + \eta)}
 \frac{g'(s)g(s)}{s - t}\,ds.
\end{aligned}
\end{equation}
In the last step, we use the lower bound on $J$ and the fact that $g^2(s) \leq C g^2(t) \leq C (s - t) $ from \eqref{eq:intdoubling}, so $1 \geq c g(s)/\sqrt{s - t}$.

The main goal of this section is to estimate the remainder of the integral, over $[t(1 + \eta), 1]$. We begin with a more careful estimate on the kernels, which shows that the positivity we just observed extends to some fixed time, independent of $t$. There is also an upper bound which we prove at the same time.

\begin{lemma}\label{lem:farfieldest} For every $0< \eta \leq 1$, and for all $t, s$ with $( 1 + \eta)t \leq  s \leq 1$,
	\[
		\left|\int_{|y| = g(s)} G(s - t, y - g(t)) - G(s, y) dS(y)\right| \leq C \frac{t}{s} g^{n-1}(s) s^{-n/2}.
	\]
	The constant $C$ depends on $\eta$ and $u$. Moreover, there exists a $T(\eta)$ (depending on $u$) such that for $( 1 + \eta)t \leq  s \leq T(\eta)$, this quantity is in fact nonnegative:
	\[
		0 \leq \int_{|y| = g(s)} G(s - t, y - g(t)) - G(s, y) dS(y) \leq C(\eta) \frac{t}{s} g^{n-1}(s) s^{-n/2}
	\]
\end{lemma}

\begin{proof}
	We proceed by estimating the heat kernels directly, in the region $s \in [(1+\eta)t, T]$:
 	\begin{align*}
		\int_{|y| = g(s)} &[G(s - t, y - g(t)) - G(s, y)] dS(y) \\
 		&= c(n) g^{n-1}(s) \int_{S^{n-1}} \left(e^{[-g^2(s)-g^2(t) + 2 g(s)g(t)y_n]/(4(s-t))} (s - t)^{-n/2} - e^{-g^2(s)/(4s)}s^{-n/2} \right)dS(y).
 	\end{align*}
	Consider the ratio
 	\begin{align*}
 		R &= \left(\frac{s - t}{s}\right)^{-n/2} \frac{1}{|S^{n-1}|e^{-g^2(s)/(4s)}}\int_{S^{n-1}} e^{[-g^2(s)-g^2(t) + 2 g(s)g(t)y_n]/(4(s-t))}dS(y) \\
		&= \left(\frac{s - t}{s}\right)^{-n/2} \text{exp}\left(- \frac{g^2(s) t }{4s(s - t)} - \frac{g^2(t)}{4(s - t)}\right)\fint_{S^{n-1}} e^{g(s)g(t)y_n/(2(s-t))} dS(y) \\
		&= R_* \fint_{S^{n-1}} e^{g(s)g(t)y_n/(2(s-t))} dS(y),
	\end{align*}
	where $R_*$ is defined via the last line. Our goal is to show that $1 \leq R \leq 1 + C \frac{t}{s}$. We will use repeatedly that $s \geq (1 + \eta) t$, so $0 \leq \frac{t}{s} \leq 1$ and $\frac{t}{s - t} \leq \frac{1}{\eta} = C(\eta)$, $1 \leq \frac{s}{s - t} \leq 1 + \frac{1}{\eta} \leq C(\eta)$. 
	
	For the mean over the sphere, $R/R_*$, the inequality $e^a \geq 1 + a$ and the fact that $y_n$ is odd imply that it is greater than $1$:
	\[
		\fint_{S^{n-1}} e^{g(s)g(t)y_n/(2(s-t))} dS(y) \geq \fint_{S^{n-1}} \left(1 + g(s)g(t)y_n/(2(s-t)) \right) dS(y) = 1.
	\]
	In the opposite direction we estimate
	\[
		\frac{g(s)g(t)|y_n|}{2(s-t)} \leq C(\eta)A_0 \frac{g(t)}{\sqrt{t}} |y_n| (t/s)^{1/2} \> \left(\leq C(\eta)A_0 \frac{g(t)}{\sqrt{t}}\right),
	\]
	using \eqref{eq:a0}, with $A_0$ the constant from there. As long as $T$ is small (in terms of $\eta$ and $u$ only, $A_0$ is fixed and depends only on $n$), we can ensure that the expression in parentheses is $\leq 1$, and so  $z \leq C(\eta)A_0 \frac{g(t)}{\sqrt{t}}$ implies $e^z \leq 1 + z + z^2$. Then estimating the integral,
	\begin{align*}
		\fint_{S^{n-1}}e^{g(s)g(t)y_n/(2(s-t))} dS(y) &\leq \fint_{S^{n-1}} 1 + \frac{g(s)g(t)y_n}{2(s-t)} + \left(\frac{g(s)g(t)y_n}{2(s-t)}\right)^2  dS(y)\\
		&\leq 1 + A_0^4 C(\eta) \frac{t}{s}\fint_{S^{n-1}} y_n^2  dS(y) \leq 1 + C(\eta)\frac{t}{s}.
	\end{align*}
	We once again used that $y_n$ is odd when integrating the linear term.

	It remains to show that $1 \leq  R_* \leq 1 + C \frac{t}{s}$. To do this, we take the logarithm:
	\[
		\log R_* = \frac{n}{2}\log\left(1 + \frac{t}{s - t}\right) - \frac{g^2(s) t }{4s(s - t)} - \frac{g^2(t)}{4(s - t)}.
	\]
 	We start with the lower bound. In the second term, we have $g^2(s)/s \leq \e^2$, where $\e$ may be made small by choosing $T$ small in terms of $\eta$:
 	\[
 		- \frac{g^2(s) t }{4s(s - t)} \geq - \e^2 \frac{1}{4}\frac{t}{s - t}.
 	\]
 	In the third term, we may similarly use that $g^2(t)/t \leq \e^2$, and so
 	\[
 		-\frac{g^2(t)}{4(s - t)} \geq - \e^2 \frac{1}{4} \frac{t}{s - t}.
 	\]
	This leads to
	\[
		\log R_* \geq \frac{n}{2}\log\left(1 + \frac{t}{s - t}\right) - \frac{\e^2}{2}  \frac{t}{s - t} = \frac{n}{2}\left[\log(1 + z) - \frac{\e^2}{n} z\right],
	\]
	where $z = \frac{t}{s - t} \in [0, \frac{1}{\eta}]$. It is clear that choosing $\e = \e(\eta)$ small, we can make this positive on $[0, \frac{1}{\eta}]$.

	For the upper bound, simply use $\log (1 + z) \leq z$, dropping all the negative terms:
	\[
		\log R_* \leq \frac{n}{2} \frac{t}{s - t} \leq C(\eta) \frac{t}{s}.
	\]
	Exponentiating and using, as above, that $t/s \leq C(\eta)$ so $e^{C t/s} \leq 1 + C(\eta)\frac{t}{s}$,
	\[
		1 \leq R_* \leq 1 + C(\eta) \frac{t}{s}.
	\]
	
	To finish the proof of the lemma, plug in the estimate on $R$ into the integral to get
	\begin{align*}
		0 &\leq \int_{|y| = g(s)} G(s - t, y - g(t)) - G(s, y) dS(y)\\
		& \leq C \int_{|y| = g(s)} \frac{t}{s} G(s, y) dS(y) \\
		&\leq C \frac{t}{s} s^{-n/2} e^{-g^2(s)/4s} g^{n-1}(s).
	\end{align*}
	The conclusion over $s \in [(1 + \eta)t, T]$ follows by dropping the exponential factor on the right, as it is $\leq 1$. 

	Over $[\eta T/(1+\eta), 1] \times B_{2A_0}$, $G$ is a smooth function. Expanding in a Taylor series, $G(s - t, y - g(t)) = G(s, y) - tG_t(s, y) - g(t) e_n \cdot \nabla G(s, y) + O(t^2 + g^2(t))$. As $G_n(s, y) = -G_n(s, -y)$, the integral over the sphere of $G_n(s, y)$ is $0$, letting us drop this term in the expansion. This leads to
	\[
		\left|\int_{|y| = g(s)} [G(s - t, y - g(t)) - G(s, y)] dS(y)\right| \leq C(n,\eta,T) \> (t + g^2(t)) g^{n-1}(s).
	\]
	Using $g^2(t) \leq A_0 t$ and $s \geq T$ completes the bound over this region.
\end{proof}

The estimate from below now follows easily.
\begin{theorem}\label{thm:past} For all $t$ small,
	$\Im_t \geq - C t.$
	If $n = 2$, then in fact
	\[
		\Im_t \geq - C t + c \int_{t + L(t)}^{t(1 + \eta)} \frac{g'(s) g(s)}{s - t} ds.
	\]
\end{theorem}

\begin{proof}
	Start with \eqref{eq:pastclose}, to get that the contribution to $\Im_t$ from $[t + L(t), (1 + \eta)t]$ is nonnegative for some $\eta$. Then apply Lemma \ref{lem:farfieldest} with this $\eta$, which gives that the contribution over $[(1 + \eta)t, T(\eta)]$ is nonnegative as well. Finally, integrate the estimate from Lemma \ref{lem:farfieldest} to bound the contribution over $[T(\eta), 1]$, which is the only part that appears in the final estimate:
 	\[
 		\int_T^1 g^{n-1}(s) g'(s) ds = \frac{1}{n}[g^n(1) - g^n(T)] \leq A_0^n,
 	\]
	so
	\[
		\int_{T}^1 \left|\int_{|y| = g(s)} [G(s - t, y - g(t)) - G(s, y)] g'(s) dS(y)\right| ds \leq C t.
	\]
	For the 2D estimate, use \eqref{eq:pastclose2D} instead as the only change.
\end{proof}

\subsection{Estimating the integrals: the future}

The final piece to estimate, $\Ip_t$, is in fact by far the easiest. It has only one of the integrals supported here, so is always negative, and we are content to almost leave it alone. The sole observation we make is that $\frac{1}{2} \leq e^{-g^2(s)/4s} \leq 1$ on $[0, t]$ if $t$ is small, so if we define
\[
	A(t) = \int_0^t g^{n-1}(s) g'(s) s^{-n/2}ds,
\]
then
\[
	- c A(t) \geq \Ip_t \geq - C A(t).
\]
We note here that as $K(0, 0)$ is absolutely convergent, so is $A(t)$, and we will use this below without further mention.

We now assemble the lower bounds over the three regions into more usable differential inequalities.
\begin{theorem}\label{thm:fullbelow} For all $t$ small,
	\[
	\frac{g^2(t)}{t} \leq C A(t).
	\]
	If $n = 2$, the stronger inequality
	\[
	\frac{g^2(t)}{t}\log \frac{t}{g^2(t)}\leq C A(t)
	\]
	holds as well.
\end{theorem}
\begin{proof}
	Recall that $|I_t|\leq C(t + g^2(t)) \leq C t$ after using that $g^2(t)/t \rightarrow 0$. Now combine with Lemma \ref{lem:transfrombelow} and Theorem \ref{thm:past} to get
	\[
	C t \geq I_t = \Ip_t + \Itrans_t + \Im_t \geq c \int_t^{t + g^2(t)}  \frac{g'(s)}{\sqrt{s - t}} ds - C t  - C A(t).
	\]
	Rearrange and replace $t$ by $\t$ to get
	\[
	\int_\t^{\t + g^2(\t)}  \frac{g'(s)}{\sqrt{s - \t}} ds \leq C [A(\t)  + \t].
	\]

	We integrate both sides in the $\t$ variable, over $[0, t]$. On the left, switch order of integration:
	\[
		\int_0^t \int_\t^{\t + g^2(\t)}  \frac{g'(s)}{\sqrt{s - \t}} ds d \tau  \geq \int_0^t \int^{s}_{m(s)}  \frac{g'(s)}{\sqrt{s - \t}} d\tau d s,
	\]
	where $m(s) + g^2(m(s)) = s$ is an inverse to the monotone function $s \rightarrow s + g^2(s)$ and we dropped any contributions with $s \geq t$. If $\tau \geq s - \frac{1}{2} g^2(s)$, use Lemma \ref{lem:doubling} to get
	\[
		g^2(s) - g^2(\tau) \leq \e (s - \t) \leq \frac{\e}{2} g^2(s)
	\]
	as long as $s \leq t$ is small. This leads to
	\[
		\t + g^2(\t) \geq s + g^2(\t) - \frac{1}{2} g^2(s) \geq s + (1 - \frac{\e}{2} - \frac{1}{2}) g^2(s) \geq s.
	\]
	In other words, $\t + g^2(\t) \geq s$, so $\t \geq m(s)$. Restricting the region of integration lets us compute the inner integral:
	\[
		\int^{s}_{m(s)}  \frac{1}{\sqrt{s - \t}} d\tau \geq \int^{s}_{s - \frac{1}{2}g^2(s)}  \frac{1}{\sqrt{s - \t}} d\tau \geq c g(s).
	\]
	The outer integral may be integrated explicitly to then get
	\[
		\int_0^t \int_\t^{\t + g^2(\t)}  \frac{g'(s)}{\sqrt{s - \t}} ds d \tau \geq c \int_0^t g(s) g'(s) ds \geq c g^2(t).
	\]
	We conclude that
	\[
		g^2(t) \leq C \int_0^t\int_0^\t g^{n-1}(s) g'(s) s^{-n/2}ds d\t  + C t^2 \leq Ct \int_0^t g^{n-1}(s) g'(s) s^{-n/2}ds  + C t^2,
	\]
	as the outer integrand on the right is monotone. This suffices for $n \geq 3$, after using \eqref{eq:strictlymonotone} to reabsorb the remainder when $t$ is small.

	When $n = 2$ we proceed in the same way, except we use the 2D form of Theorem \ref{thm:past} to get
	\[
		C t \geq c \int_{t + g^2(t)}^{(1 + \eta) t}  \frac{g(s) g'(s)}{s - t} ds - C t  - C \int_0^t g^{n-1}(s) g'(s) s^{-n/2}ds,
	\]
	dropping the positive term from $\Itrans_t$ in favor of the one from $\Im_t$. Now proceed in essentially the same way, except now integrating over $\t \in [\frac{1}{1 + \eta}s, s - g^2(s)] \ss [\frac{1}{1+\eta} s, m(s)]$. The containment holds because $s = m(s) + g^2(m(s)) \leq m(s) + g^2(s)$. We get
	\begin{align*}
		\int_0^t \int_{\t + g^2(\t)}^{(1 + \eta)\t}  \frac{g(s) g'(s)}{s - \t} ds d \tau  &\geq \int_0^t \int^{s - g^2(s)}_{s/(1 + \eta)}  \frac{g(s) g'(s)}{s - \t} d\tau d s \\
		& \geq \int_0^t  g(s) g'(s) \log \left(\frac{\eta s}{(1 + \eta)g^2(s)} \right) d s 
		\geq c \int_0^t  g(s) g'(s) \log \frac{s}{g^2(s)} ds,
	\end{align*}
	the last step being true for $g^2(s)/s$ sufficiently small. Now integrate by parts (on $[a, t]$ and then take a limit using $g^2(t)/t \rightarrow 0$) to get
	\[
		\int_0^t  g(s) g'(s) \log \frac{s}{g^2(s)} ds \geq \frac{1}{2} g^2(t) \log \frac{t}{g^2(t)} - \frac{1}{2}\int_0^t \frac{g^2(s)}{s} ds,
	\]
	estimating the remainder by
	\[
		\int_0^t \frac{g^2(s)}{s} ds \leq t \int_0^t \frac{g^2(s)}{s^2} ds \leq 2 t \int_0^t \frac{g(s) g'(s)}{s} ds.
	\]
	Combining everything leads to
	\[
		\frac{g^2(t)}{t}\log \frac{t}{g^2(t)}  \leq C \int_0^{t} g(s) g'(s) s^{-1}ds  + C t
	\]
	and we conclude as before.
\end{proof}
This is one of the nonlinear differential inequalities for $g$ we were seeking.

Rather than prove just an upper bound like $A(t) \leq C g^2(t)/t$, we can at this point actually compute the precise limit of the ratio with a more careful analysis of the potential bounds as well as an extra time averaging.

\begin{theorem}\label{thm:asymptotic}
	If $n \geq 3$, then
	\[
		\lim_{t \searrow 0} \frac{t A(t)}{g^2(t)} = \frac{(4\pi)^{n/2}}{2 (n - 2)\w_{n-1}},
	\]
	where $\w_{n-1}$ is the volume of $S^{n-1}$. If $n = 2$, then instead
	\[
		\lim_{t \searrow 0} \frac{t A(t)}{g^2(t)} \frac{1}{\log \frac{t}{g^2(t)}} = \frac{1}{2}.
	\]
\end{theorem}

\begin{proof}
	We begin by observing, from Theorem \ref{thm:fullbelow} and Lemma \ref{lem:doubling}, that
	\[
		t A'(t) = t (t^{-n/2} g^{n-1}(t) g'(t)) \leq C t^{-n/2} g^n(t) \leq C A(t) \left( \frac{g^2(t)}{t} \right)^{(n-2)/2},
	\]
	so
	\[
		\lim_{t \rightarrow 0} \frac{t A'(t)}{A(t)} = 0,
	\]
	when $n \geq 3$. When $n = 2$ we instead get
	\[
		t A'(t) \leq C  \frac{A(t)}{\log (t/g^2(t))},
	\]
	which leads to the same conclusion.

	Integrating this leads to
	\[
		\left|\lim_{t \searrow 0} \log \frac{A(\lambda t)}{A(t)}\right| = \left|\lim_{t \searrow 0} \int_{t}^{\lambda t} \frac{A'(s)}{A(s)} ds\right| \leq  |\log \lambda| \limsup_{s \searrow 0} \frac{sA'(s)}{A(s)} = 0
	\]
	for any $\lambda \in (0, \infty)$, so in particular $A(\lambda t)/A(t) \rightarrow 1$. As $A$ is monotone and $A(\lambda t) < A(t)$ for $0<\lambda \leq 1$, applying the dominated convergence theorem and changing variables leads to
	\begin{equation}\label{eq:precise3}
		\frac{1}{t A(t)}\int_0^t A(s) ds = \int_0^1 \frac{A(\lambda t)}{A(t)} d\lambda \rightarrow 1.
	\end{equation}
	A final consequence of this differential inequality is that for $t < s$ small relative to $\e$, integrating $A'(t)/A(t) \leq \e/t$ leads to $A(s) s^{-\e} \leq A(t) t^{-\e}$, and in particular
	\begin{equation}\label{eq:precise1}
		\lim_{t \searrow 0} \frac{t}{A(t)} = 0.
	\end{equation}

	We use these observations to refine our potential estimates. We have the exact relation
	\[
		I_t = \int_t^{2t} J(g(s), g(t), s - t) g'(s) ds - \frac{\w_{n-1}}{(4\pi)^{n/2}} \int_0^{2t} e^{-g^2(s)/4s} g^{n-1}(s) g'(s) s^{-n/2} ds + R_t,
	\]
	where $R$ is the contribution from the past:
	\[
		R(t) = \int_{2t}^1 \int_{|y| = g(s)} G(s - t, y - g(t)) - G(s, y) dS(y) g'(s) ds.
	\]

	We use Lemma \ref{lem:farfieldest} to bound $R$:
	\[
		|R(t)| \leq C t \int_{2t}^1 g^{n-1}(s) g'(s) s^{-n/2 - 1} ds \leq C t \int_{2t}^1 \frac{A'(s)}{s} ds.
	\]
	We claim that for every $\e > 0$, there is a $t(\e)$ such that for $t < t(\e)$, $|R(t)| \leq \frac{\e}{10} A(t)$. Indeed, integrate over $[2t, T]$ and $[T, 1]$ separately, choosing $T$ so that $s A'(s) \leq \d A(s)$ for $s < T$:
	\begin{align*}
		t \int_{2t}^T \frac{A'(s)}{s} ds &\leq t \int_{2t}^T \d \frac{A(s)}{s^2} ds \\
		&\leq t^{1 -\d} A(t) \int_{2t}^T \d s^{\d-2} ds 
		\leq \frac{\d}{1 - \d} A(t).
	\end{align*}
	Choose $\d$ such that $C \frac{\d}{1 - \d} < \frac{\e}{20}$, in the process choosing $T$. Then for the remaining part of the integral,
	\[
		t \int_{T}^1 \frac{A'(s)}{s} ds \leq C t \int_{T}^1 g^{n}(s) s^{-n/2 - 2} ds \leq C(T) t,
	\]
	using $g(s)\leq A_0$. Select $t(\e)$ small enough so that $C(T) t \leq \frac{\e}{20} A(t)$ for $t < t(\e)$, from \eqref{eq:precise1}.

	Together with the fact that $|I(t)| \leq C t \leq \frac{\e}{10} A(t)$ for $t$ small, we have that
	\[
		\left|\int_t^{2t} J(g(s), g(t), s - t) g'(s) ds - \frac{\w_{n-1}}{(4\pi)^{n/2}} \int_0^{2t} e^{-g^2(s)/4s} g^{n-1}(s) g'(s) s^{-n/2} ds\right| \leq \frac{\e}{5} A(t).
	\]
	In the second term on the left, observe that $e^{-g^2(s)/4s}\rightarrow 1$ as $s \rightarrow 0$, so after further restricting $t$ and also using that $A(2t)/A(t) \rightarrow 1$, we get
	\[
		\left|\int_t^{2t} J(g(s), g(t), s - t) g'(s) ds - \frac{\w_{n-1}}{(4\pi)^{n/2}} A(t)\right| \leq \frac{\e}{2} A(t).
	\]
	As this is true for any $\e$ for sufficiently small $t$, we have shown that
	\[
		\lim_{t \searrow 0} \frac{1}{A(t)}\int_t^{2t} J(g(s), g(t), s - t) g'(s) ds = \frac{\w_{n-1}}{(4\pi)^{n/2}}.
	\]
	Rather than estimate this integral directly, we first average one more time in $t$:
	\begin{equation}\label{eq:precise2}
		\lim_{M\searrow 0}\frac{1}{A(M)M} \int_0^M \int_t^{2t} J(g(s), g(t), s - t) g'(s) ds dt = \lim_{M\searrow 0} \frac{1}{A(M)M} \int_0^M A(t) \frac{\w_{n-1}}{(4\pi)^{n/2}} dt = \frac{\w_{n-1}}{(4\pi)^{n/2}},
	\end{equation}
	using the dominated convergence theorem.

	We now estimate this time-averaged quantity, changing order of integration:
	\begin{align*}
		\int_0^M &\int_t^{2t} J(g(s), g(t), s - t) g'(s) ds dt \\
		&= \int_0^M \int_{s/2}^s g'(s) J(g(s), g(t), s - t) dt ds + \int_M^{2M} \int_{s/2}^M g'(s) J(g(s), g(t), s - t) dt ds \\
		&=: E_1 + E_2.
	\end{align*}
	We first show that $E_2$ is small, and then carefully compute the limit of $E_1$.

	Apply Lemma \ref{lem:transfromabove} to the integrand of $E_2$: on the region of integration, $t \leq s \leq 2 t$, so $g^2(t) \leq g^2(s) \leq C g^2(t)$ by \eqref{eq:intdoubling}. We therefore obtain  $J(g(s), g(t), s - t) \leq C \min\{(s - t)^{-1/2}, g^{n-1}(t) (s - t)^{-n/2}\}$. Now integrate, using the change of variables $z = s - t$:
	\begin{align*}
		E_2 &\leq C \frac{g(M)}{M} \int_0^{M} \int_{\max(z,M-z)}^M \min\{z^{-1/2}, g^{n-1}(t) z^{-n/2}\} dt dz \\
		&\leq  C \frac{g(M)}{M} \int_0^{M} z \min\{z^{-1/2}, g^{n-1}(M) z^{- n/2}\} dz \\
		&\leq C \frac{g(M)}{M} \left[\int_{0}^{g^2(M)} z^{1/2} dz + g^{n-1}(M) \int_{g^2(M)}^{M}  z^{1 - n/2} dz \right].
	\end{align*}
	We used that $s \approx t \approx M$ and $g^2(t) \approx g^2(s) \approx g^2(M)$ repeatedly, as well as Lemma \ref{lem:doubling}. After evaluating the integrals, the two terms give:
	\[
		E_2 \leq C g^2(M) \begin{cases}
			g^2(M)/M + 1, & n = 2 \\
			g^2(M)/M + g(M)/\sqrt{M}, & n = 3 \\
			g^2(M)/M + g^2(M)/M \log (M/g^2(M)), & n = 4\\
			g^2(M)/M,  & n \geq 5.
		\end{cases}
	\]
	The important point here is that when $n \geq 3$, $E_2/g^2(M) \rightarrow 0$, while when $n = 2$, $E_2 \leq C g^2(M)$.

	To compute $E_1$, the key point is to evaluate the limit of the inner integral
	\[
		\int_{s/2}^s J(g(s), g(t), s - t) dt = \int_0^{s/(2 g^2(s))} g^{-1}(s) J(1, \frac{g(s - \t g^2(s))}{g(s)}, \t)  g^2(s) d \t
	\]
	after using a change of variables $s - t = g^2(s) \t$ and noting that $J(R, R', a) = R^{-1} J(1, R'/R, a/R^2)$. Now, for small $s$ we have that 
	\[
		g^2(s) - g^2(s - \t g^2(s))  = \int_{s - \t g^2(s)}^s 2 g(z) g'(z) dz \leq C \t \frac{g^2(s)}{s} g^2(s)
	\]
	from Lemma \ref{lem:doubling}, and so 
	\[
 c\leq \frac{g(s-\tau g^2(s))}{g(s)}\leq1,
 \qquad
 \frac{g(s-\tau g^2(s))}{g(s)}\longrightarrow1
 \quad\text{for each fixed }\tau>0.
\]
In particular, $J(1,\frac{g(s - \t g^2(s))}{g(s)}, \t ) \rightarrow J(1, 1, \t)$ for every $\t$ as $s \rightarrow 0$. For $n \geq 3$, use the bound $J(1,\frac{g(s - \t g^2(s))}{g(s)}, \t ) \leq C \min\{\t^{-1/2}, \t^{-n/2} \} \in L^1([0, \infty))$ from Lemma \ref{lem:transfromabove} with the dominated convergence theorem to conclude that
	\[
		\frac{1}{g(s)}\int_{s/2}^s J(g(s), g(t), s - t) dt = \int_0^{s/(2 g^2(s))}  J(1, \frac{g(s - \t g^2(s))}{g(s)}, \t) d \t \rightarrow \int_0^{\infty}  J(1, 1, \t) d \t.
	\]
	This integral can be evaluated directly:
	\begin{align*}
		\int_0^{\infty}  J(1, 1, \t) d \t &= \int_{S^{n-1}} \int_0^\infty (4 \pi \t)^{-n/2}e^{-|Y - e_n|^2/(4\t)} d\t  dS(Y)\\
		&= \int_{S^{n-1}} (4\pi)^{-n/2} \Gamma(\frac{n}{2} - 1) 2^{n -2} |Y - e_n|^{2 - n}  dS(Y) \\
		&= \frac{1}{(n - 2)\w_{n-1}}\int_{S^{n-1}}|Y - e_n|^{2 - n} dS(Y),
	\end{align*}
	using the volume of the unit sphere. Changing to polar coordinates,
	\begin{align*}
		\int_{S^{n-1}}|Y - e_n|^{2 - n} dY &= \w_{n-2} \int_{0}^\pi [2 \sin(\theta/2)]^{2 -n} \sin^{n-2} \theta d\theta \\
		&= \w_{n-2} \int_{0}^\pi \cos^{n-2}(\theta/2) d\theta \\
		&= 2 \w_{n-2} \int_{0}^{\pi/2} \cos^{n-2}(\theta) d \theta \\
		&= \int_{S^{n-1}}1 dS(y) \\
		&= \w_{n-1},
	\end{align*}
	where the second-to-last step changes variables back to the sphere in a different parametrization.

	We have shown that, if $n \geq 3$,
	\begin{align*}
		E_1 &= \int_0^M \int_{s/2}^s  J(g(s), g(t), s - t) dt g'(s) ds \\
		&= \int_0^M (\frac{1}{n - 2} + o_M(1))g(s) g'(s) ds = (\frac{1}{n - 2} + o_M(1)) \frac{g^2(M)}{2}.
	\end{align*}
	Therefore $\frac{E_1}{g^2(M)} \rightarrow \frac{1}{2(n-2)}$, while $\frac{E_2}{g^2(M)} \rightarrow 0$. Plugging into \eqref{eq:precise2} gives
	\[
		\lim_{M \searrow 0} \frac{M A(M)}{g^2(M)} = \frac{(4\pi)^{n/2}}{2 (n-2) \w_{n-1}},
	\]
	and so we conclude.

	For $n = 2$, we compute $E_1$ more directly, still focusing on the inner integral
	\[
		\int_{s/2}^s J(g(s), g(t), s - t) dt = \int_{s/2}^{s - g^2(s)} J(g(s), g(t), s - t) dt + \int_{s - g^2(s)}^s J(g(s), g(t), s - t) dt.
	\]
	The second integral turns out to be small:
	\[
		\int_{s - g^2(s)}^s J(g(s), g(t), s - t) dt \leq C \int_{s - g^2(s)}^s (s - t)^{-1/2} dt \leq C g(s),
	\]
	using Lemma \ref{lem:transfromabove}. For the other part, set $\t := s - t$ and estimate (using $g(t)\leq g(s)$)
	\[
		J(g(s), g(s - \t), \t) = \frac{g(s)}{2 \t}\fint_{S^{1}} e^{-|g(s)Y - g(s-\tau)e_2|^2/(4\t)} dY = \frac{g(s)}{2 \t}(1 + O(\frac{g^2(s)}{\t})).
	\]
	That $1$ is the leading-order contribution: if we integrate in $\t$,
	\[
		\int_{s/2}^{s - g^2(s)} J(g(s), g(t), s - t) dt = \int_{g^2(s)}^{s/2} J(g(s), g(s -\t), \t) d\t = \frac{g(s)}{2}[\log \frac{s}{g^2(s)} + O(1)].
	\]
	Here we absorbed the $O(g^2(s)/s)$ remainders into $O(1)$. Thus
	\[
		E_1 = \int_0^M \int_{s/2}^s J(g(s), g(t), s - t) dt g'(s) ds = \int_0^M \frac{g(s)g'(s)}{2}[\log \frac{s}{g^2(s)} + O(1)] ds.
	\]
	Integrating by parts (on $[a, M]$ and then taking a limit),
	\[
		E_1 = O(g^2(M)) + \frac{1}{4} g^2(M) \left[\log \frac{M}{g^2(M)} + 1\right] - \frac{1}{4} \int_0^M \frac{g^2(s)}{s} ds,
	\]
	and so
	\[
		E_1 + E_2 = O(g^2(M) + \int_0^M \frac{g^2(s)}{s} ds) + \frac{1}{4} g^2(M) \log \frac{M}{g^2(M)}
	\]
	Using Theorem \ref{thm:fullbelow}, we know that 
	\[
 \ell_*(M):=\inf_{0<s\leq M}\log\frac{s}{g^2(s)}
 \longrightarrow\infty,
\]
\begin{align*}
 \frac1{MA(M)}\int_0^M\frac{g^2(s)}{s} \> ds
 &\leq \frac{C}{MA(M)\ell_*(M)}
       \int_0^M A(s) \> ds\\
 &\leq \frac{C}{\ell_*(M)}
 \longrightarrow0,
\end{align*}
using the monotonicity of $A$.

Now returning to \eqref{eq:precise2},
	\[
		\frac{1}{2} = \lim_{M\searrow 0}\frac{1}{A(M)M} \int_0^M \int_t^{2t} J(g(s), g(t), s - t) g'(s) ds dt = \lim_{M\searrow 0}  \frac{g^2(M)}{4 M A(M)}  \log \frac{M}{g^2(M)}.
	\]
	This gives the $n = 2$ version of the conclusion after taking reciprocals.
\end{proof}

\section{ODE analysis}\label{sec:ode}

In this section, we study the asymptotic differential inequalities (starting with $n \geq 3$)
\[
	A(t) = \frac{(4\pi)^{n/2}}{2 (n - 2)\w_{n-1}} \frac{g^2(t)}{t}(1 + o_t(1)),
\]
which come from Theorem \ref{thm:asymptotic}.

Our goal is to obtain bounds on $g$. We begin with a change of variables:
\[
	t = e^{-\tau}, \qquad h(\tau) = \frac{g^2(t)}{t} = e^\tau g^2(e^{-\tau}).
\]
Dots as in $\dot{h}$ refer to derivatives with respect to $\tau$. Note for later that in these variables, we know that $h(\tau) \rightarrow 0$ as $\tau \rightarrow \infty$. Then
\[
	\dot{h} = h - 2 g g' \qquad \implies \qquad g g' = \frac{h - \dot{h}}{2}.
\]
In these variables $A(e^{-\tau}) = \int_0^{e^{-\tau}} g^{n-1}(s) g'(s) s^{-n/2}ds$; rewriting in terms of $h$ and changing variables leads to
\[
	A(e^{-\tau}) = \int_\tau^\infty \frac{h - \dot{h}}{2} h^{\frac{n-2}{2}} e^{\tau'}  e^{-\tau'}d\tau' = \int_\tau^\infty \frac{h - \dot{h}}{2} h^{\frac{n-2}{2}}d\tau'.
\]
In this notation we have
\[
	A = \frac{(4\pi)^{n/2}}{2 (n - 2)\w_{n-1}} h (1 + o_\t(1)).
\]

It will be helpful to integrate the expression
\[
	A(e^{-\tau}) = \int_\tau^\infty \frac{h - \dot{h}}{2} h^{\frac{n-2}{2}}d\tau'
\]
by parts to express $A$ purely in terms of $h^{n/2}$ and its primitive. While this is an absolutely convergent integral and $A(e^{-\tau}) \rightarrow 0$ as $\tau \rightarrow \infty$, it is not immediately clear that when written as a difference of two terms they are each absolutely convergent (we have not shown yet that $\dot{h} \leq 0$, in particular). To avoid problems we first integrate by parts on an interval and then take a limit of the result.
\[
	A(e^{-\tau}) - A(e^{-s}) = \frac{1}{2} \left[\int_\tau^s h^{n/2} d\tau' - \int_\tau^s \dot{h} h^{n/2 - 1} d\tau'\right] = \frac{1}{2}\int_\tau^s h^{n/2} d\tau' - \frac{1}{n} h^{n/2}(s) + \frac{1}{n} h^{n/2}(\tau).
\]
Taking $s \rightarrow \infty$, and using that $h(s) \rightarrow 0$ on the right, shows that $\int_\tau^\infty h^{n/2} d\tau' < \infty$ and
\[
	A(e^{-\tau}) = \frac{1}{2}\int_\tau^\infty h^{n/2} d\tau' +  \frac{1}{n} h^{n/2}(\tau).
\]
Set
\[
	B(\tau) = \int_\tau^\infty h^{n/2} d\tau';
\]
then we have $A = \frac{1}{2} B - \frac{1}{n} \dot{B}$ and $-h^{n/2} = \dot{B}$. 

\begin{theorem}
	Assume $n \geq 3$. Then
	\[
	B(\tau) = \left(\frac{n - 2}{2}\right)^{- \frac{2}{n - 2}} \left(\frac{(4\pi)^{n/2}}{(n - 2)\w_{n-1}}\right)^{\frac{n}{n-2}} \tau^{- \frac{2}{n - 2}} (1 + o_\t(1)).
	\]
	We also have
	\[
		g(t) = \left(\frac{2 (4\pi)^{n/2}}{(n - 2)^2\w_{n-1}}\right)^{\frac{1}{n-2}} \sqrt{t}( - \log t)^{-\frac{1}{n - 2}} (1 + o_t(1))
	\]
	and
	\[
		g' \leq C \frac{1}{\sqrt{t}} ( - \log t)^{- \frac{1}{n - 2}}.
	\]
\end{theorem}

\begin{proof}
	Start with
	\[
		A = \frac{(4\pi)^{n/2}}{2 (n - 2)\w_{n-1}} h (1 + o_\t(1)).
	\]
	We rewrite $A = \frac{1}{2} B + \frac{1}{n} h^{n/2}$, so
	\[
		\frac{1}{2} B + \frac{1}{n} h^{n/2} = \frac{(4\pi)^{n/2}}{2 (n - 2)\w_{n-1}} h (1 + o_\t(1)).
	\]
	Recalling that $h \rightarrow 0$, the middle term with $h^{n/2}$ can be reabsorbed to give
	\[
		B =\frac{(4\pi)^{n/2}}{(n - 2)\w_{n-1}} h (1 + o_\t(1)).
	\]
	Now use $h^{n/2} = - \dot{B}$ and integrate the resulting differential inequalities
	\begin{equation}\label{eq:Bineq}
		- B^{-\frac{n}{2}} \dot{B} = \left(\frac{(4\pi)^{n/2}}{(n - 2)\w_{n-1}}\right)^{-n/2} (1 + o_\t(1))
	\end{equation}
	from a large $\t_*$ to $\t$ to get
	\[
		B^{-n/2 + 1}(\t) = \frac{n - 2}{2} \left(\frac{(4\pi)^{n/2}}{(n - 2)\w_{n-1}}\right)^{-n/2} \t (1 + o_\t(1)) ,
	\]
	or
	\begin{equation}\label{eq:Bc1}
		B = \left(\frac{n - 2}{2}\right)^{- \frac{2}{n - 2}} \left(\frac{(4\pi)^{n/2}}{(n - 2)\w_{n-1}}\right)^{\frac{n}{n-2}} \tau^{- \frac{2}{n - 2}} (1 + o_\t(1)).
	\end{equation}

	Plugging \eqref{eq:Bc1} into \eqref{eq:Bineq} gives
	\begin{equation}\label{eq:Bpc1}
		 - \dot{B} = \tau^{- \frac{n}{n - 2}} \left(\frac{(4\pi)^{n/2}}{(n - 2)\w_{n-1}}\right)^{\frac{n}{n-2}} \left(\frac{n - 2}{2}\right)^{- \frac{n}{n - 2}}(1 + o_\t(1)) ,
	\end{equation}
	so from the definition of $B$ we see that
	\[
		h = \left(\frac{n - 2}{2}\right)^{- \frac{2}{n - 2}} \left(\frac{(4\pi)^{n/2}}{(n - 2)\w_{n-1}}\right)^{\frac{2}{n-2}}\tau^{- \frac{2}{n - 2}}(1 + o_\t(1))  .
	\]
	Rewriting in terms of $g$,
	\[
		\frac{g(t)}{\sqrt{t}( - \log t)^{-\frac{1}{n - 2}}} \rightarrow \left(\frac{n - 2}{2}\right)^{- \frac{1}{n - 2}} \left(\frac{(4\pi)^{n/2}}{(n - 2)\w_{n-1}}\right)^{\frac{1}{n-2}} =  \left(\frac{2 (4\pi)^{n/2}}{(n - 2)^2\w_{n-1}}\right)^{\frac{1}{n-2}}.
	\]
	The upper bound on $g'$ follows directly from Lemma \ref{lem:doubling}.
\end{proof}

There is a corresponding result in 2D, although the ODE has different behavior due to the logarithmic corrections.
\begin{theorem}
	Assume $n = 2$. Then
	\[
	B = e^{-\sqrt{2\t} + o(\sqrt{\t})}.
	\]
	We also have
	\[
		g(t) = \sqrt{t} e^{-\sqrt{-\frac{1}{2}\log t} + o(\sqrt{-\log t})}
	\]
	and
	\[
		g' \leq \frac{1}{\sqrt{t}} e^{-c\sqrt{-\log t}}.
	\]
\end{theorem}

\begin{proof}
	We define $h, A, B$ in exactly the same way, and $A = \frac{1}{2} B + \frac{1}{n} h^{n/2}$ remains valid. The asymptotic relation is instead
	\[
		 B +  h = h \log \frac{1}{h} (1 + o_\t(1)),
	\]
	and reabsorbing gives
	\[
		 B = h \log \frac{1}{h} (1 + o_\t(1)).
	\]
	Here it helps to take another logarithm, $w = - \log B$, to get
	\[
		\dot{w} = \frac{1}{\log 1/h}(1 + o_\t(1)).
	\]
	using $h = - \dot{B}$. As just taking logs of the inequality for $B$ instead gives 
	\[
		w = -\log B = - \log h - \log (-\log h) + o(1) = - \log h (1 + o_\t(1))
	\]
	after reabsorbing the constant and double-log terms, this can be expressed as
	\[
		\dot{w} = \frac{1}{w}(1 + o_\t(1)).
	\]
	Integrating leads to
	\[
		w = \sqrt{2 \t} (1 + o_\t(1))
	\]
	and so
	\[
		B = e^{-\sqrt{2\t} + o(\sqrt{\t})}.
	\]
	Rewriting in terms of $g$ gives the other conclusions:
	\[
		h = B \dot{w} = e^{-\sqrt{2\t} + o(\sqrt{\t})} \frac{1}{\sqrt{\t}} = e^{-\sqrt{2\t} + o(\sqrt{\t})},
	\]
	so $g(t) = \sqrt{t} e^{-\sqrt{-\frac{1}{2}\log t} + o(\sqrt{-\log t})}$. The bound on $g'$ follows from Lemma \ref{lem:doubling}, after reabsorbing the constant into the exponential.
\end{proof}

\section{Appendix}\label{sec:appendix}
\subsection{A gradient estimate}
\begin{lemma}[Boundary gradient estimate with bounded drift]
\label{lem_append}
Let $B\geq0$ and let $b$ be a bounded continuous function
on $(0,1]\times(0,1)$ with $|b|\leq B$.
Suppose that $f$ is continuous on $[0,1]\times[0,1]$,
is a classical solution of
\[
    f_a-f_{yy}+b(a,y)f_y=0
    \qquad\text{in }(0,1]\times(0,1),
\]
and has one-sided spatial derivatives at $y=0$ for
positive times. Assume that
\[
    0\leq f\leq M,
    \qquad f(a,0)=0\quad(0\leq a\leq1).
\]
Then
\[
    0\leq f_y(a,0+)\leq C(B)M
    \qquad\text{for }3/4\leq a\leq1.
\]
The constant depends only on $B$, and not on any
derivatives or modulus of continuity of $b$.
\end{lemma}

\begin{proof}
Set
\[
    F(a,y):=\operatorname{erf}
       \left(\frac{y}{2\sqrt a}\right),
    \qquad
    E(a,y):=\exp\left(-\frac{By}{2}
                        +\frac{3B^2a}{4}\right),
\]
where
\[
    \operatorname{erf}(z)
    :=\frac{2}{\sqrt\pi}\int_0^z e^{-\xi^2}\,d\xi.
\]
Define
\[
    K_B:=\frac{e^{B/2}}{\operatorname{erf}(1/2)},
    \qquad W(a,y):=K_BM E(a,y)F(a,y).
\]
Since $F_a=F_{yy}$, direct computation gives
\[
    (\partial_a-\partial_{yy}+b\partial_y)W
    =
    K_BM E
    \left[
        \frac{B(B-b)}{2}F+(B+b)F_y
    \right]
    \geq0.
\]
Moreover, $W(a,0)=0$, and, for $0<a\leq1$,
\[
    W(a,1)\geq
    K_BM e^{-B/2}\operatorname{erf}(1/2)=M.
\]
For every fixed $y\in(0,1]$,
\[
    \lim_{a\searrow0}W(a,y)
    =K_BM e^{-By/2}\geq M.
\]
At the corner $(0,0)$, continuity of $f$, the identity
$f(0,0)=0$, and $W\geq0$ imply
\[
    \limsup_{\substack{(a,y)\to(0,0)\\a,y>0}}
    (f-W)(a,y)\leq0.
\]
The parabolic comparison principle therefore gives
$f\leq W$. Taking right derivatives at $y=0$, where
$f=W=0$, yields
\[
    0\leq f_y(a,0+)
    \leq W_y(a,0+)
    =\frac{K_BM e^{3B^2a/4}}{\sqrt{\pi a}}.
\]
For $3/4\leq a\leq1$, the right-hand side is bounded by
\[
    \frac{2e^{B/2+3B^2/4}}
         {\sqrt{3\pi}\operatorname{erf}(1/2)}M,
\]
which proves the claim.
\end{proof}
\subsection{Lipschitz regularity of the free boundary}
Here we sketch the proof of Lipschitz regularity of the free boundary satisfying the assumptions (1-4).
The idea is to transform the radial solution to a 1D solution of a Stefan problem with H\"older coefficients.

In this subsection $t$ denotes physical time. We write
\[
 \gamma(t):=g(-t),\qquad -1<t<0,
 \qquad W(t,x):=x^{(n-1)/2}u(t,x).
\]
Thus the positive phase is $\{x>\gamma(t)\}$. Work on a compact
time interval on which the interface remains a positive distance
from $x=0$ and from the chosen spatial endpoints.

Then $W$ is a solution in the sense of distributions of
$$\left(W + a(r) \chi_{\{W>0\}}\right)_t = W_{rr} + c(r) W,$$
where $a(r) = r^{(n-1)/2}$ and $c(r)=-(n-1)(n-3)/(4r^2)$ are locally H\"older continuous 
and $a$ is strictly positive in $r>0$.

We modify the regularity theory of the 1D Stefan problem in \cite{FPK} by A. Fasano, M. Primicerio and S. Kamin
to the equation satisfied by $W$:

Standard energy estimates yield
$$W_t \in L^2, W_x \in L^\infty(L^2) \text{ locally in } r>0.$$
As $W_{xx} = W_t - c W$ in each phase, we obtain also that
the restriction of $W_{xx}$ to each phase is in $L^2$ locally in $r>0$.

Let now for some time interval $J$ 
$$\gamma(t)\in [\ell+\delta,r-\delta], \text{ where } 0<\ell<\ell+\delta \leq \gamma(t) \leq r-\delta < r.$$
Then 
$$B(t)^2 := \int_\ell^{\gamma(t)} |W_{xx}(t,\cdot)|^2 dx + \int_{\gamma(t)}^r |W_{xx}(t,\cdot)|^2 dx$$
satisfies 
$$\int_J B(t)^2 dt < \infty.$$
Let $I_+(t) = (\gamma(t),r)$ and $I_-(t) = (\ell, \gamma(t))$ and choose
for a.e. $t$ points $x^0_\pm$ such that
$$|W_x(t,x^0_\pm)|^2 \leq \frac{1}{|I_\pm(t)|} \int_{I_\pm(t)} |W_x(t,\cdot)|^2 dx.$$
It follows that
\begin{align*}
 &|W_x(t,x)|^2 \leq |W_x(t,x^0_\pm)|^2 + 2 \int_{I_\pm(t)} |W_x(t,\cdot)W_{xx}(t,\cdot)| dx\\ &\quad\quad
 \leq \frac{1}{|I_\pm(t)|} \int_{I_\pm(t)} |W_x(t,\cdot)|^2 dx + 2 \Vert W_x(t,\cdot)\Vert_{L^2(I_\pm(t))}\Vert W_{xx}(t,\cdot)\Vert_{L^2(I_\pm(t))},\\
 &\Vert W_x(t,\cdot)\Vert_{L^\infty(I_\pm(t))}^2 \leq \delta^{-1} \Vert W_x\Vert_{L^\infty(L^2)}^2 + 2 \Vert W_x\Vert_{L^\infty(L^2)} B(t),\\
  &\Vert W_x(t,\cdot)\Vert_{L^\infty(I_\pm(t))}^4 \leq 2\delta^{-2} \Vert W_x\Vert_{L^\infty(L^2)}^4 + 8 \Vert W_x\Vert_{L^\infty(L^2)}^2 B^2(t),\\
  &\int_J\Vert W_x(t,\cdot)\Vert_{L^\infty(I_\pm(t))}^4 dt \leq C |J| \delta^{-2} \Vert W_x\Vert_{L^\infty(L^2)}^4 + C \Vert W_x\Vert_{L^\infty(L^2)}^2  \int_J B^2(t) dt <\infty.
\end{align*}
From the equation for $W$ in the sense of distributions we obtain now after approximation
with smooth cutoffs that
$$\frac{d}{dt} \left( \int_\ell^r W(t, \cdot) dx - A(\gamma(t))\right) = W_x(t,r) - W_x(t,\ell) + \int_\ell^r c W dx$$
in the sense of distributions,
where $A(y)=\int_\ell^y a(x) dx$ is the primitive of $a$.
From the energy estimates we infer that $A(\gamma(t))$ is absolutely continuous which implies that $\gamma$ is absolutely continuous,
that
\begin{equation}\label{flux_eq} a(\gamma(t)) \gamma'(t) = [-W_x(t,\gamma(t))],\end{equation}
and by the above $L^4$-estimate, that $\gamma'$ is locally in $L^4$.

Next, we straighten the interface, setting $y=x-\gamma(t)$ and $Z(t,y)=W(t,y+\gamma(t))$, and obtaining that $Z(t,0)=0$
and that $$Z_t-Z_{yy} = \gamma'(t) Z_y + c(y+\gamma(t)) Z \text{ in } \{ y<0 \}\cup \{ y>0 \},$$
where the second derivative is taken in the sense of distributions.
We already know that $\gamma' \in L^4$ and $Z_y \in L^4(L^\infty)$ locally,
so the right-hand side of this PDE is in $L^2(L^\infty)$ locally.

Choose cutoffs $\eta$ and $\psi$ such that $\eta=0$ near and before the initial time
and $\psi=0$ near and right to the spatial endpoint,
and extend $Z$ by $V := \eta(t) \psi(y) Z$ to $y>0$.
The resulting right-hand side $F$ is still in $L^2(L^\infty)$.
Using the heat kernel for the half line $y>0$,
$$G_H(t,y,z) := \frac{1}{\sqrt{4\pi t}} \left(e^{-(y-z)^2/(4t)}- e^{-(y+z)^2/(4t)}\right),$$
we obtain for $t_0< t < t_1 <0$
$$V(t,y) = \int_{t_0}^t \int_0^\infty G_H(t-s,y,z) F(s,z) dz ds \text{ for } y,z>0,$$
where
$$ \sup_{y>0} \int_0^\infty |\partial_y G_H(t,y,z)| dz \leq C t^{-1/2}.$$
It follows that
$$ \Vert V_y(t,\cdot)\Vert_{L^\infty} \leq C \int_{t_0}^t (t-s)^{-1/2}  \Vert F(s,\cdot) \Vert_{L^\infty} ds.$$
Consequently, setting $k(t) := t^{-1/2}$, noting that $5/8 + 1/2 = 1 + 1/8$ and using Young's convolution inequality on a finite interval, we obtain that 
$$\Vert V_y\Vert_{L^8(L^\infty)} \leq C \Vert k\Vert_{L^{8/5}(0,t_1-t_0)} \Vert F\Vert_{L^2(L^\infty)}. $$
Hence $Z_y \in L^8(L^\infty)$ locally.
Choosing secondary cutoffs with support in the regions where $\eta$ and $\psi$ equal $1$, we obtain that
the new right-hand side $F_2 \in L^{8/3}(L^\infty)$.
H\"older's inequality implies that the new $V_2$ satisfies
$V_{2,y}\in L^\infty$ locally, which implies $Z_y\in L^\infty$ locally,
which proves the desired boundedness of $W_x$ in the positive phase.
Repeating the argument for $\tilde Z(t,y)=W(t,\gamma(t)-y)$ and combining
the two estimates, we obtain that
$W_x\in L^\infty$ locally.
Applying once more \eqref{flux_eq}, we obtain that $\gamma'\in L^\infty$ locally
which proves the local Lipschitz continuity. 

\section*{AI Use}

GPT 5.4-6 was used to obtain preliminary versions of some of the computations in this paper, including the estimates on the heat kernel differences used in Section \ref{sec:potential} and some of the barrier constructions in Section \ref{sec:pde}. The estimates actually appearing in the paper have been heavily simplified and revised in places, and all of the text is written by the authors. The core idea of the paper, to turn the potential representation of the solution into a differential equation for the free boundary with enough information to capture the fine asymptotics at a singularity, had no AI input. The authors take responsibility for the correctness of all the results presented.

\section*{Acknowledgments}

DK was supported by NSF DMS grants 2247096 and 2554417.

\bibliographystyle{plainurl}
\bibliography{stefan_arxiv}

\end{document}